\documentclass[12pt]{amsart}
\usepackage{a4wide, amsmath, amssymb}
\newtheorem{thm}{Theorem}[section]

\newtheorem{lem}[thm]{Lemma}
\newtheorem{prop}[thm]{Proposition}
\theoremstyle{definition}
\theoremstyle{definition}
\theoremstyle{definition}
\newtheorem{ex}[thm]{Example}\theoremstyle{definition}
\newtheorem{defn}[thm]{Definition}
\newtheorem{rem}[thm]{Remark} \numberwithin{equation}{section}

\newcommand{\A}{\mathcal{A}}
\newcommand{\C}{\mathbb{C}}
\newcommand{\D}{\mathcal{D}}
\newcommand{\F}{\mathcal{F}}
\newcommand{\Hc}{\mathcal{H}}

\newcommand{\R}{\mathbb R}
\newcommand{\E}{\mathbb E}
\newcommand{\T}{\mathbb T}
\newcommand{\N}{\mathbb N}
\def\P{\mathbb P}
\def\1{\mathbb I}
\def\a{\alpha}
\def\b{\beta}

\def\e{\epsilon}
\def\g{\gamma}
\def\l{\lambda}

\def\o{\omega}

\def\eps{\varepsilon}

\def\ov{\overline}
\begin{document}
\title[Large time behavior of weakly coupled systems]{Large time behavior of weakly coupled systems of first-order
Hamilton-Jacobi equations}
\author{Fabio Camilli,  Olivier Ley, Paola Loreti \and Vinh Duc Nguyen}
\address{Dipartimento di Scienze di Base e Applicate per l'Ingegneria,  ``Sapienza" Universit{\`a}  di Roma,
 00161 Roma, Italy}
\email{e-mail:camilli@dmmm.uniroma1.it}
\address{IRMAR, INSA de Rennes, 35708 Rennes, France} \email{olivier.ley@insa-rennes.fr}
\address{Dipartimento di Scienze di Base e Applicate per l'Ingegneria,  ``Sapienza" Universit{\`a}  di Roma,
 00161 Roma, Italy}
\email{e-mail:loreti@dmmm.uniroma1.it}
\address{IRMAR, INSA de Rennes, 35708 Rennes, France} \email{vinh.nguyen@insa-rennes.fr}
\begin{abstract}
We show a large time behavior result for class
of  weakly coupled systems of first-order
Hamilton-Jacobi equations in the periodic setting.
We use a PDE approach to extend the convergence result
proved by Namah and Roquejoffre (1999) in the scalar case.
Our proof is based on new comparison, existence and regularity results
for systems. An interpretation of the solution of the system
in terms of an optimal control problem with switching is given.
\end{abstract}


\maketitle

\section{introduction}

The
aim of this paper is to study the large time behavior of the system of Hamilton-Jacobi equations
\begin{eqnarray}\label{HJEi}
&& \left\{
\begin{array}{ll}
\displaystyle\frac{\partial u_i}{\partial t}
+  H_i(x, Du_i )+\sum_{j=1}^{m}d_{ij}(x)u_j=0 &
(x,t)\in \T^N\times (0,+\infty),\\[3mm]
u_i(x,0)=u_{0i}(x) & x\in\T^N,
\end{array}
\right.
 i=1,\dots,m,
\end{eqnarray}
where $\T^N$ is the $N$-dimensional torus. The Hamiltonians $H_i(x,p)$ are of eikonal type and
the coupling is linear and monotone, i.e.,
\begin{eqnarray}\label{coupling-intro}
&& d_{ii}(x)\geq 0, \quad d_{ij}(x)\leq 0 \ {\rm for} \ i\not= j\quad {\rm and}  \quad
\sum_{j=1}^{m}d_{ij}(x)\geq 0, \qquad \text{ for all } x\in\T^N.
\end{eqnarray}

The corresponding problem for the Hamilton-Jacobi equation
\begin{equation}\label{HJi}
    \frac{\partial u}{\partial t}+  H(x, Du )=0
\end{equation}
has been extensively investigated using both PDE methods, see
Namah and Roquejoffre \cite{nr99}, Barles and Souganidis \cite{bs00}, and
a dynamical approach: Fathi \cite{fathi98, fathi10}, Roquejoffre \cite{roquejoffre01},
Davini and Siconolfi \cite{ds06}. Some of these results have been also extended 
beyond the periodic setting: Barles and Roquejoffre \cite{br06}, Ishii \cite{ishii08},
Ichihara and Ishii \cite{ii09}
and for problems with periodic boundary conditions: see for instance 
Mitake~\cite{mitake08a, mitake08b, mitake09}. We refer also the readers to
Ishii~\cite{ishii06, ishii09} for an overview.
\smallskip

In these works, one of the main result is 
that there exists a constant $c\in\R, $
the so-called critical value or ergodic constant, 
and a solution $v$ of the stationary equation
\begin{equation}\label{sHJi}
   H(x, Du )=c
\end{equation}
such that
\begin{eqnarray}\label{cv-scalaire}
u(x,t)+ct\mathop{\to} v(x) \quad \text{ uniformly as } t\to +\infty.
\end{eqnarray}
There are several equivalent characterizations of the critical value (see \cite{bs00}, \cite{fathi10}), for example $c$ is the unique constant such that
\begin{eqnarray}\label{cv-to-erg}
\frac{u(x,t)}{t}\to -c \quad \text{ uniformly as } t\to +\infty,
\end{eqnarray}
or
\begin{equation}\label{eci}
c=\min\{a\in\R:\, H(x,Du)=a\, \text{has a subsolution}\}.
\end{equation}
While $c$ is uniquely determined, the main difficulty in proving
a result like~\eqref{cv-scalaire} is that~\eqref{sHJi} does not admit
a unique solution (at least, the equation is invariant by addition
of constants).
\smallskip

To our knowledge, there are not only  no results of asymptotic 
type for the system \eqref{HJEi}, but also
the study of corresponding ergodic problem
\begin{equation}\label{statioHJi}
     H_i(x, Dv_i(x) )+\sum_{j=1}^{m}d_{ij}(x)v_j(x)=c_i
\qquad x\in\T^N,
\end{equation}
is not  well understood (see \cite{cl08} for some preliminary results).
\smallskip

With the  aim of understanding if some convergence like~\eqref{cv-scalaire}
holds in the case of systems, we focus on the setting of Namah and Roquejoffre \cite{nr99}.
 Let us start by recalling the main result
of \cite{nr99}. It takes place in the periodic setting
and they assume that $H$ in \eqref{HJi} is continuous and of the type
\begin{eqnarray*}
H(x,p)=F(x,p)-f(x) \qquad x\in\T^N, p\in\R^N,
\end{eqnarray*}
where $F$ is coercive and convex with respect to $p.$
Besides, $F(x,p)\geq F(x,0)=0.$ The function $f$ is continuous and
satisfies
\begin{eqnarray}\label{F-scalaire}
f\geq 0 \quad {\rm and} \quad \F_{\rm scalar}
=\{x\in\T^N:\,f(x)=0\}\not=\emptyset.
\end{eqnarray}
It is simple to see by  the characterization in \eqref{eci} that $c=0$. 
Moreover by classical results in viscosity solution theory, $\F_{\rm scalar}$ is a uniqueness set for \eqref{sHJi}, i.e. the solution of \eqref{sHJi} is uniquely characterized by its value on this set. The coercitivity of the Hamiltonian  provides the compactness of the functions $u(\cdot,t)$ for $t>0$. Then employing the semi-relaxed limits, one can pass to the
limit and obtain the convergence result if one can prove the convergence
of $u(\cdot,t)$ on the set $\F_{\rm scalar}$. This latter result follows from the observation
that, since $F(x,p)\ge 0,$ a solution of \eqref{HJi} satisfies
\begin{eqnarray*}
\frac{\partial u}{\partial t}\leq 0 \qquad
{\rm on} \ \F_{\rm scalar}\times (0,+\infty).
\end{eqnarray*}
Hence $u(\cdot,t)$ is nonincreasing and therefore converges uniformly on $\F_{\rm scalar}$ and one concludes that it converges  in all $\T^N$.
\smallskip

Our purpose is to reproduce the previous proof and therefore we assume that
the Hamiltonians $H_i$'s in \eqref{HJEi} are as in \cite{nr99}
and the coupling matrix $D(x)=(d_{ij})_{1\leq i,j\leq m}$
satisfies~\eqref{coupling-intro}. Under these assumptions,~\eqref{HJEi}
has a unique viscosity solution in $\T^N\times [0,+\infty)$ for any
continuous initial data $u_0:\T^N\to \R^m.$
\smallskip

For the simplicity of the exposition in this introduction, we assume
moreover that
\begin{eqnarray}\label{hyp-simple-D}
\sum_{j=1}^m d_{ij}(x)=0,\ i=1,\dots,m, \qquad
\sum_{i=1}^m d_{ij}(x)=0,\ j=1,\dots,m,
\end{eqnarray}
for $x\in\T^N$.
Note that this assumption is not necessary. We can avoid it by using
the results of Section~\ref{s2}.
\smallskip

To continue, we have to understand what plays the role of $\F_{\rm scalar}$
in~\eqref{F-scalaire} for systems.
In the scalar case,
there are important interpretations of the convergence~\eqref{cv-scalaire}
in terms of dynamical systems or optimal control theory    \cite{fathi10}.
Indeed,~\eqref{cv-scalaire} means that the optimal trajectories
of the related control problem are attracted by the set $\F_{\rm scalar}$
where the running cost $f$ is 0.\smallskip

 In the case of systems,
the solutions $u_i$'s of~\eqref{HJEi}
are value functions of a piecewise deterministic optimal control
problem with random switchings. The switchings are governed by a continuous in time
Markov process with state space $\{1,\dots, m\}$ and
probability transitions from the mode $i$ to $j$ given by
$\gamma_{ij}=-d_{ij}$ for $i\not= j.$ See Section~\ref{sec:control}
for further details.
A natural assumption to obtain the convergence of the value functions
$u_i$'s is to require that all the running costs $f_i$'s vanish at least at some
common point. It suggests that the optimal strategy consists in driving
the trajectories to such a point where the running costs are 0 whatever
the switchings will be (note that the coercivity of the Hamiltonian
implies the controllability of the trajectories). So we introduce
and assume that
\begin{eqnarray}\label{intro-F}
\F:=\bigcap_{i=1}^m \{x\in \T^N: f_i(x)=0\}
=\{x\in \T^N: \sum_{i=1}^m f_i(x)=0\}
\not= \emptyset
\end{eqnarray}
(recall that the $f_i\geq 0$).
We need an additional assumption on the coupling matrix,
namely that $D(x)$ is irreducible, see Definition~\ref{def-irr}.
Roughly speaking, it means that the coupling is not trivial and
the system cannot be reduced to several subsystems of lower dimensions.
\smallskip

The next step is to understand well the limit problem \eqref{statioHJi}.
Under the previous  assumptions,
$\F$ appears to be a uniqueness set for the stationary system
\eqref{statioHJi} with $c_i=0$ (as in~\cite{nr99}, we will prove that
our assumptions imply that the critical value is $c=0$).
More precisely, on this
set it is sufficient to control the value of the sum
$v_1+\dots+v_m$ of a solution to \eqref{statioHJi}, see Theorem~\ref{confr}, a  condition which seems to be new
with respect to standard assumptions on weakly coupled systems (\cite{el91}, \cite{ik91}).
Let us mention that, when~\eqref{hyp-simple-D} does not hold, we need to replace
$\F$ with another set $\A,$ see~\eqref{F}-\eqref{aubry}.
\smallskip

We then solve the so-called ergodic problem, which consists
in finding a couple $(c,v)\in \R^m\times C(\T^N;\R^m)$
of solutions to~\eqref{statioHJi}.
The motivation comes from the
formal expansion suggested by the convergence result
of type~\eqref{cv-scalaire} we are expecting for~\eqref{HJEi}.
Plugging $u_\infty(x)-ct\approx u(x,t)$ in~\eqref{HJEi}, we obtain
that $(c,u_\infty)$ should be a solution of
$$
H_i(x, D(u_\infty)_i)+\sum_{j=1}^m d_{ij}(x)(u_\infty)_j (x)
-t \sum_{j=1}^m d_{ij}(x)c_j= c_i,
$$
for all $i,$ with $c\in {\rm ker}(D(x))$ to cancel the term in $t.$
In Theorem~\ref{thm-ergodic}, we prove the existence of a solution
where $c\in {\rm ker}(D(x))$ and $v$ Lipschitz continuous on $\T^N.$
The Lipschitz continuity of $v$ is an easy consequence of the coercivity
of the $F_i$'s. Under the assumptions~\eqref{hyp-simple-D},
it is easy to see that
${\rm ker}(D(x))$ reduces to the line spanned by $(1,\dots,1)$ so
$c=(c_1,\dots,c_1).$ Moreover, due to~\eqref{intro-F}, we obtain
that $c=(0,\dots,0)$ is uniquely determined.
\smallskip

At this step, it is worth noticing that we can solve the ergodic problem
in a more general setting (see Theorem~\ref{thm-ergodic-bis}), in particular
without assuming~\eqref{intro-F}. We obtain the following condition on
$c,$
\begin{eqnarray*}
c=(c_1,\dots, c_1)
\quad {\rm and} \quad
\sum_{i=1}^m \mathop{\rm min}_{\T^N}f_i(x)
\leq -c_1 \leq
\mathop{\rm min}_{\T^N} \sum_{i=1}^m f_i(x).
\end{eqnarray*}
This gives again an indication that assuming $\F\not=\emptyset$
is a first natural case to consider, since, in this case, inequalities
are replaced with equalities in the above formula and the ergodic constant
is univocally defined.
\smallskip

We are now in force to consider the large time result (Theorem~\ref{LTB}).
The coercivity of the Hamiltonians and the existence of a solution
to the ergodic problem
give the compactness of the sequences $u_i$'s in $W^{1,\infty}(\T^N\times [0,+\infty)).$
An easy consequence is the
convergence~\eqref{cv-to-erg} for all $i$ to 0 (since $c=0$ in our case).
To mimic the proof of~\cite{nr99}, we need to prove the convergence
of the $u_i$'s on $\F.$ This is the most difficult part of the work.
\smallskip

Indeed, by summing the equations~\eqref{HJEi} for $i=1,\dots ,m,$ we obtain
\begin{eqnarray}\label{somm-eqs}
\sum_{i=1}^m \frac{\partial u_i}{\partial t}
+\sum_{i=1}^m H_i(x,Du_i)
+\sum_{i,j=1}^m d_{ij}u_j =0.
\end{eqnarray}
Using that $H_i(x,Du_i)\geq 0$ on $\F$ and~\eqref{hyp-simple-D},
we obtain easily that
\begin{eqnarray}\label{deriv-somme-u}
\frac{\partial}{\partial t}\sum_{i=1}^m u_i(x,t)\leq 0
\qquad {\rm on} \ \F
\end{eqnarray}
and therefore
$t\mapsto (u_1+\dots+u_m)(\cdot,t)$ is nonincreasing  and converges
uniformly as $t\to +\infty$ on $\F.$
But this is not enough to prove the convergence of {\it each}
$u_i$ on $\F.$ 
\smallskip

To overcome this difficulty, we use some ideas 
of~\cite{bs00}. We choose a subsequence $t_n\to +\infty$
such that $u(\cdot, t_n+\cdot)$ converges uniformly
to some $w(\cdot,\cdot)$ in $W^{1,\infty}(\T^N\times [0,+\infty)).$
By stability of the viscosity solutions, $w$ is still solution
of~\eqref{HJEi} and we earn something: now, $t\mapsto (w_1+\dots+w_m)(\cdot,t)$ is constant
on $\F.$ Therefore~\eqref{deriv-somme-u} holds for the $w_i$'s
with an equality. It follows from~\eqref{somm-eqs} that
\begin{eqnarray*}
\sum_{i=1}^m H_i(x,Dw_i)=0 \qquad {\rm on } \ \F\times (0,+\infty).
\end{eqnarray*}
Since $H_i=F_i\geq 0$ on $\F,$ we infer that 
$H_i(x,Dw_i)=0$ for all $1\leq i\leq m.$
Therefore, the system~\eqref{HJEi} reduces to a linear differential
system 
\begin{eqnarray*}
\frac{\partial w}{\partial t}(x,t)+D(x)w(x,t)=0
\qquad t\geq 0,
\end{eqnarray*}
for every $x\in\F.$ Using that $D(x)$ 
satisfies~\eqref{coupling-intro},~\eqref{hyp-simple-D} and
is irreducible, we can prove the convergence of each $w_i(\cdot,t)$ on $\F$
and then on $\T^N$ by applying the comparison Theorem~\ref{confr}.
The conclusion follows by proving that
$u(\cdot,t)$ converges to the same limit as $w(\cdot,t).$
\smallskip

The paper is organized as follows. In Section \ref{s2}, we study some
properties of the coupling matrix $D$ without assumption~\eqref{hyp-simple-D}.
Section \ref{s3} is devoted to study existence and  uniqueness of the stationary problem. In Section \ref{s4} we solved the ergodic problem. The convergence
result is proved in
Section \ref{s5}. Finally, in Section \ref{sec:control}, we give a control theoretic interpretation of the problem.
\smallskip

We learnt recently that Mitake and Tran~\cite{mt11} studied systems
of two equations ($m=2$) both in our setting (see Remark~\ref{rmq-thm-cv} (3))
and also in some particular cases related to~\cite{bs00}.
\smallskip

\noindent{\sc Notation.} If $p=(p_1,\dots,p_m)$ is a vector in $\R^m,$
then $p\geq 0$ (respectively $p>0$) means that $p_i\geq 0$ (respectively
$p_i>0$) for $1\leq i\leq m.$
\smallskip

\noindent{\sc Acknowledgments.} O.L. is partially supported by the project ANR BLANC07-3 187245, ``Hamilton-Jacobi and Weak KAM Theory''. We would like to
thank G.~Barles and P.~Cardaliaguet who helped us to improve the
first version of this work and M.~Briane, L.~Herv\'e
and J.~Ledoux for useful references and suggestions.

\section{Preliminaries on coupling matrices}
\label{s2}
We consider the matrix $D(x)=(d_{ij}(x))_{1\leq i,j\leq m}$ and call it a
{\it coupling matrix} for the systems~\eqref{HJEi} and~\eqref{statioHJi} .
We assume that the coupling matrices satisfy the following standard assumptions (see \cite{ishii92}, \cite{ik91})
\begin{equation}\label{H3}
\begin{split}
&\text{$d_{ij}:\T^N\to \R$ are continuous and, for all $x\in\T^N,$}  \\
& d_{ii}(x)\geq 0, \quad d_{ij}(x)\leq 0 \ {\rm for} \ i\not= j\quad {\rm and}  \quad
\sum_{j=1}^{m}d_{ij}(x)\geq 0.
\end{split}
\end{equation}
We introduce some conditions on the matrix  $D$ we will be interested in:

\begin{defn}\label{def-irr}
We say that $D$ is
\begin{itemize}
\item[(i)]  a \emph{$M$-matrix} if
\begin{eqnarray*}
&& D=sI-B, \quad\text{for some } s>0, \quad B=(b_{ij})_{1\leq i,j\leq m}, \ b_{ij}\geq 0,\\
&&\text{with } s\geq \rho(B) \text{ and } \rho(B)  \text{ the spectral radius of } B,
\end{eqnarray*}
($I$ is the identity matrix).
\item[(ii)]  {\em irreducible} if,
for all subset $\mathcal{I}\varsubsetneq \{1,\cdots ,m\}$ then there exists $i\in \mathcal{I}$ and  $j\not\in  \mathcal{I}$
 such that $d_{ij}\not= 0$.
\end{itemize}
\end{defn}
\noindent

\begin{lem}\label{mmat}
If $D$ satisfies~\eqref{H3}, then it is a $M$-matrix for any $x$.
\end{lem}
\begin{proof}
If $D=0,$ there is nothing to prove. Otherwise,
$s:={\rm max}_{1\leq k\leq m}d_{kk} >0.$ Then, we can write $D=sI-B$ with
$B= (b_{ij})$ with $b_{ii}=s-d_{ii}$ and $b_{ij}=-d_{ij}$ for $i\not=j.$
Since $B\geq 0,$ by Perron-Froebenius theorem, the spectral radius $\rho(B)$
of $B$ is an eigenvalue and there exists a nonnegative eigenvector $p$
such that $Bp=\rho(B)p.$ Therefore $Dp=(s-\rho(B))p.$
Let $p_k={\rm max}_{i} p_i.$ Since $B\not=0,$ we have $p_k>0$
and, using that $d_{ij}\leq 0$ for $i\not=j,$
\begin{eqnarray*}
0\leq \left(\sum_{j=1}^m d_{kj}\right)p_k\leq \sum_{j=1}^m d_{kj}p_j=(s-\rho(B))p_k
\end{eqnarray*}
and we conclude that $s\geq \rho(B).$
\end{proof}
Let us give a characterization of an irreducible matrix.
\begin{lem}
$D$ is irreducible if and only if:
for all $i,j\in  \{1,\cdots ,m\},$ there exists
$n\in\N$ and a sequence $i_0=i, i_1, i_2, \cdots , i_n=j$
such that $d_{i_{l-1} i_{l}}\not= 0$ for all $1\leq l\leq n$
(in this case we  say that there exists a chain between $i$ and $j$).
\end{lem}
\begin{proof}
Let $i\in \{1,\cdots ,m\}$ and $\mathcal{I}_i$ be the subset
of  $\{1,\cdots ,m\}$ containing all the chains starting from $i.$
It is obvious that, if $D$ is irreducible, then
$\mathcal{I}_i= \{1,\cdots ,m\}.$ Conversely,
let $\mathcal{I}\varsubsetneq \{1,\cdots ,m\},$ $i\in \mathcal{I}$
and $j\not\in \mathcal{I}.$ By assumption, there exists a chain
$i_0=i, i_1, \cdots , i_n=j$ between $i$ and $j.$ Let $1\leq \bar{l}\leq n$
be the smallest $l$ such that $l\not\in \mathcal{I}.$ Then
$d_{i_{ \bar{l}-1 } i_{\bar{l}} }\not= 0.$
\end{proof}

\begin{ex}
If $d_{ij}\not=0$ for all $i,j$ then
it is obvious that $D$ is irreducible. In particular,
when $D$ satisfies \eqref{H3} and, in addition, $d_{ij}<0$ for $i\not= j,$
then it is irreducible. The matrices
\begin{eqnarray*}
\left[
\begin{array}{ccc}
 1 & -1 & 0 \\
 0 & 1 & -1 \\
 -1 & 0 & 1
\end{array}
\right]
\quad {\rm and} \quad
\left[
\begin{array}{cc}
  \a & -\a \\
  -\b & \b\ \\
\end{array}
\right] \ \ \text{with } \a, \b >0
\end{eqnarray*}
are irreducible and satisfy \eqref{H3}.
On the contrary,
\begin{eqnarray}
\label{nonirred}
\left[
\begin{array}{cccc}
 1 & -1 & 0 & 0 \\
 -1 & 1 & 0 & 0 \\
 0 & 0 & 1 & -1 \\
 0 & 0 & -1 & 1
\end{array}
\right]
\quad {\rm and} \quad
\left[\begin{array}{cc}
  0 & 0 \\
  -1 & 1\ \\
\end{array}\right]
\end{eqnarray}
satisfy \eqref{H3} but are not irreducible.
\end{ex}

\begin{rem} For $M$-matrices and irreducible $M$-matrices, see
\cite[Chapter 6]{bp94} and~\cite{seneta81}.
These assumptions are natural when studying coupled systems
of partial differential equations like~\eqref{statioHJ} and~\eqref{HJE}.
To expect some maximum principles, one usually needs $M$-matrices
and, roughly speaking, when the coupling matrix is irreducible, it means
that the equations are coupled in a non trivial way.
For instance, in the cases~\eqref{nonirred},
the system of equations is decoupled (into two subsets of equations) and
is triangular, respectively. We refer to the work of Busca and Sirakov~\cite{bs04}
and the references therein for details.
\end{rem}

\begin{lem} \label{lem-rang}
Let $E_1, E_2\subset\T^N$ be closed subsets, let $E=E_1\cap E_2,$
suppose that  $D(x)=(d_{ij}(x))_{1\leq i,j\leq m}$ is such that~\eqref{H3} holds
and
\begin{eqnarray}
&& \sum_{j=1}^m d_{ij}(x) =0 \text{ for } x\in E_1,
\qquad i=1,\dots,m. \label{degenerate}\\
&& D(x) \text{ is irreducible for } x\in E_2,\label{Dcst-irr}
\end{eqnarray}
\begin{itemize}
\item[(i)]
For all $x\in E,$ $D(x)$
is degenerate of rank $m-1,$
${\rm ker}(D(x))={\rm span}\{(1, \cdots, 1)\}$
and the real part of each nonzero complex eigenvalue of
$D(x)$ is positive. Moreover
there exists a positive continuous function
$\Lambda=(\Lambda_1,\dots,\Lambda_m): E\to\R^m$ such that
$\Lambda(x) >0$ and $D(x)^T\Lambda(x) =0$ for all $x\in E.$

\item[(ii)]
In the set $E_2\setminus E_1,$
where one only has $\displaystyle\sum_{j=1}^m d_{ij}(x) \geq 0,$
then there exists a positive continuous function
$\Lambda=(\Lambda_1,\dots,\Lambda_m): E_2\setminus E_1\to\R^m$ such that
$\Lambda(x) >0$ and $D(x)^T\Lambda(x) \geq 0$ for all $x\in E_2\setminus E_1.$

\end{itemize}

\end{lem}
\begin{proof}
(i) At first, we fix any $x\in E$ (this part of the
proof comes from~\cite[p.156]{bp94}). From~\eqref{degenerate},
$(1,\cdots,1)\in {\rm ker}(D(x)).$
From Lemma~\ref{mmat} and~\eqref{Dcst-irr}, $D(x)$ is a nonzero
$M$-matrix so we can write
$D(x)=s(x)I-B(x),$ $s(x)>0,$ $s(x)\geq \rho(B(x))$ and $B(x)\geq 0.$
It is obvious that, if $D(x)$
is irreducible, then so is $B(x)$.
By Perron-Froebenius theorem, it follows that
$\rho(B(x))$ is a simple eigenvalue of $B(x).$
Since, $D(x)(1,\cdots,1)=0=s(x)(1,\cdots,1)-B(x)(1,\cdots,1),$
we have $s(x)=\rho(B(x))$ and 0
is a simple eigenvalue of $D(x).$ Thus $D(x)$ has rank $m-1.$
Besides, if $\lambda\in\C\setminus\{0\}$ is another eigenvalue of
$D(x),$ then $s(x)-\lambda$ is an eigenvalue of $B(x).$
It follows that $|s(x)-\lambda|\leq \rho(B(x))=s(x)$ 
and, since $\lambda\not= 0,$
the real part of $\lambda$ must be positive.

Since $B^T(x)$ is also an irreducible nonnegative matrix
with $\rho(B(x))=\rho(B^T(x)),$ using again
Perron-Froebenius theorem, we obtain the existence of
an eigenvector $\Lambda(x) >0$
such that $B^T(x)\Lambda(x) =\rho(B(x))\Lambda(x).$
Therefore $D^T(x)\Lambda(x)=0.$

Then, we shall prove that it is possible to choose $\Lambda(x)$
continuously. Let ${\rm com}(D(x))$ be the cofactor matrix of $D(x).$
Since $D(x)$ is non invertible, we have
$D(x)^T {\rm com}(D(x))=0.$ Therefore the columns $C_j(x),$
$j=1,\dots,m,$ of ${\rm com}(D(x))$
are in the kernel of $D(x)^T.$
From the first part of the proof, we obtain that
there exist functions $\lambda_i: E \to\R$ such that
$C_j(x)=\lambda_j(x)\Lambda(x).$
Define $|C_j|(x)$ as the absolute value of the coefficients of
$C_j(x)$ and
$$
\tilde{\Lambda}(x)=\sum_{j=1}^m |C_j|(x)=\left(\sum_{j=1}^m|\lambda_j(x)|
\right)\Lambda(x).
$$
On  one hand, $\tilde{\Lambda}(x)>0$ since ${\rm com}(D(x))$ is not
0. On the other hand,
the continuity of the coefficients of $D(x)$ implies the continuity
of the coefficients of ${\rm com}(D(x))$ and therefore the maps
$x\mapsto |C_j|(x)$ are continuous on $E.$ We conclude that
$\tilde{\Lambda}$ is continuous.\\
(ii) The proof is an easy consequence of (i). Define
$\alpha_i(x):= \sum_{j=1}^m d_{ij}(x)\geq 0,$
the diagonal matrix $\Delta(x):= {\rm diag}(\alpha_1(x),\cdots,\alpha_m(x))$
and $\tilde D:= D-\Delta.$ It is straightforward that $\tilde D$
is still an irreducible $M$-matrix 
such that~\eqref{degenerate}
holds on $E_2\setminus E_1.$ By (i),  there exists a continuous
$\Lambda: E_2\setminus E_1 \to\R^m$ such that
$\Lambda(x) >0$ and $\tilde D(x)^T\Lambda(x) = 0.$
It follows $D(x)^T\Lambda(x)=\Delta(x)^T \Lambda(x)\geq 0$
by assumption. It completes the proof.
\end{proof}

%


\section{Comparison, existence and  regularity
for the stationary system}\label{s3}
In this section we study existence and uniqueness of the solution to
 the stationary system
\begin{eqnarray}\label{statioHJ}
\displaystyle
H_i(x,Du_i)+ \sum_{j=1}^m d_{ij}(x)u_i=0\quad \text{in } \T^N,
\quad 1\leq i\leq m,
\end{eqnarray}
where $H_i:\T^N\times\R^N\to\R$, $i=1,\dots,m$, is a continuous function which
takes the form
\begin{eqnarray}\label{Hroquejo}
H_i(x,p)=F_i(x,p)-f_i(x).
\end{eqnarray}
 We assume that, for all $i=1,\cdots,m,$
\begin{eqnarray}
   & & \text{$f_i, F_i(\cdot,p)$ are continuous $1$-periodic for any $p\in\R^N$;}\label{H0}\\
   & & \text{$F_i(x,\cdot)$ is convex, coercive and,  for any $x\in \T^N,$
$p\in\R^N,$ }
  F_i(x,p)\geq F_i(x,0)=0\; ;\label{H1} \\
 & & f_i(x)\geq 0. \label{regf}
\end{eqnarray}
We set for $i=1,\dots,m,$
\begin{eqnarray}
&\F=\{x\in \T^N:\sum_{i=1}^m f_i(x)=0\},\qquad\D_i=\{x\in \T^N: \sum_{j=1}^{m}d_{ij}(x)=0\} \label{F}\\
&\A=\F\cap\left( \bigcap_{i=1}^m \D_i\right).\label{aubry}
\end{eqnarray}

\begin{rem}\label{minfi}
Note that, under~\eqref{regf}, if $\F$ is not empty, it means that
all the $f_i$'s achieve a common minimum 0
at some common point.
\end{rem}

We recall the definition of viscosity solutions for the  system  \eqref{statioHJ} (see \cite{ishii92}, \cite{ik91} for more details about
systems of Hamilton-Jacobi equations). Let $USC$ (respectively $LSC$) denotes the upper-semicontinuous (respectively lower-semicontinuous) functions.

\begin{defn}
\item[(i)] An $USC$ function $u:  \R^N \to\R^m$
is said a viscosity subsolution of \eqref{statioHJ}  if
whenever $\phi\in C^1$,
$i\in \{1,\dots,m\}$ and $u_i-\phi$ attains
a local maximum at $x$, then
\[H_i(x,D\phi(x))+\sum_{j=1}^{m}d_{ij}(x)u_j(x)\le 0.\]
\item[(ii)]  A $LSC$ $u:  \R^N  \to\R^m$ is said a viscosity  supersolution of
\eqref{statioHJ}  if whenever $\phi\in C^1$,
$i\in \{1,\dots,m\}$ and $u_i-\phi$ attains
a local minimum at $x$, then
\[ H_i(x,D\phi(x ))+\sum_{j=1}^{m}d_{ij}(x)u_j(x) \ge 0.\]
\item[(iii)] A continuous function $u$  is said a viscosity solution of  \eqref{statioHJ} if it is both
a viscosity sub- and supersolution of \eqref{statioHJ}.
\end{defn}


We first prove a comparison theorem for \eqref{statioHJ} giving a boundary condition on the set
\eqref{aubry}, which turns out to be a uniqueness set for the system.

\begin{thm}\label{confr}
Assume \eqref{H3} and \eqref{H0}--\eqref{regf}.
Let $u\in USC(\T^N)$ and $v\in LSC(\T^N)$ be respectively a bounded
subsolution  and a bounded supersolution  of \eqref{statioHJ} and suppose that
one of the following set of assumptions holds:
\begin{itemize}
\item[(i)] Classical case:
\begin{eqnarray}\label{strict-pos}
\sum_{j=1}^{m}d_{ij}(x)>0 \quad \text{in } \T^N \ \text{ for all } 1\leq i\leq m.
\end{eqnarray}

\item[(ii)] Degenerate case:
Assume \eqref{Dcst-irr} holds with $E_2=\T^N$ and there exists
\begin{eqnarray}
&&\ \Lambda:\A\to \R^m, \ \Lambda \geq 0, \
 \sum_{i=1}^m \Lambda_i >0, \text{ such that }\\
&& \sum_{i=1}^m \Lambda_i(x) u_i(x)\leq   \sum_{i=1}^m \Lambda_i(x) v_i(x), \
x\in\A.\label{phiA}
\end{eqnarray}
\end{itemize}
Then
\[u\le v\quad \text{ in $\T^N$.}\]
\end{thm}

\begin{proof}
The proof of the classical case can be deduced from the lines
of the degenerate case, so we skip it and turn to the degenerate case.
See some comments at the end of Case 2 below.

Let $0<\mu<1,$ and consider
\begin{eqnarray}\label{supbla}
\mathop{\rm sup}_{\T^N}\ \mathop{\rm sup}_{1\leq k\leq m} \{ \mu u_k - v_k\}=: M_\mu.
\end{eqnarray}
We assume that $M_\mu> 0$ (otherwise, there is nothing to prove).
By compactness, the above maximum is achieved for some $k_0$ at some $x_0\in \T^N.$
We set
$$
\mathcal{I}=\{k\in\{1,\cdots, m\}: (\mu u_k - v_k)(x_0)=M_\mu\}.
$$
We distinguish 3 cases.\\
{\it Case 1: $\mathcal{I}=\{1,\cdots, m\}$ and $x_0\in \mathcal{A}.$}
For all $k,$ we get
$$
\sum_{k=1}^m\Lambda_k(x_0) M_\mu
= \sum_{k=1}^m\Lambda_k(x_0)(\mu u_k - v_k)(x_0)
\leq (\mu-1)  \sum_{k=1}^m\Lambda_k(x_0)v_k(x_0)
\leq (1-\mu)\sum_{k=1}^m\Lambda_k(x_0)|v|_\infty
$$
and therefore $M_\mu\leq (1-\mu)|v|_\infty.$
\medskip

\noindent{\it Case 2: $\mathcal{I}=\{1,\cdots, m\}$ but $x_0\not\in \mathcal{A}.$}
One can find $i\in\{1,\dots,m\}$ such that
\begin{equation}\label{iiii}
   \text {either}\qquad f_i(x_0)>0\quad \text{or}  \qquad \sum_{j=1}^{m}d_{ij}(x_0)>0.
\end{equation}
Consider
\begin{eqnarray*}
\mathop{\rm sup}_{\T^N\times \T^N} \{ \mu u_i(x) - v_i(y)-\frac{|x-y|^2}{2\e^2}-|x-x_0|^2\}.
\end{eqnarray*}
The latter maximum is greater than $M_\mu$ and is
achieved at some $(\bar{x},\bar{y})$ which satisfy the following
classical properties:
\begin{eqnarray}\label{estim-class}
\bar{x},\bar{y}\to x_0 \quad \text{and} \quad
\frac{|\bar{x}-\bar{y}|^2}{2\e^2},|\bar{x}-x_0|^2 \to 0 \quad
\text{as } \e\to 0.
\end{eqnarray}
We set $\displaystyle p_\e= \frac{\bar{x}-\bar{y}}{\e^2}.$

Writing that $u_i$ is a viscosity subsolution of \eqref{statioHJ}, we have
\begin{eqnarray}\label{ineg-soussol}
\mu F_i(\bar{x},\frac{p_\e+2(\bar{x}-x_0)}{\mu})
+\sum_{j=1}^m d_{ij}(\bar{x})\mu u_j(\bar{x})
\leq \mu f_i(\bar{x})
\end{eqnarray}
and writing that $v_i$ is a supersolution of \eqref{statioHJ}, we get
\begin{eqnarray*}
F_i(\bar{y},p_\e)
+\sum_{j=1}^m d_{ij}(\bar{y}) v_j(\bar{y})
\geq f_i(\bar{y}).
\end{eqnarray*}

From the coercivity of $F_i$ and the boundedness of $f_i$ and the
$d_{ij}$'s on $\T^N,$ \eqref{ineg-soussol} implies that
\begin{eqnarray}\label{Cgrad}
p_\e\leq C=C(F_i,D,f_i).
\end{eqnarray}

We subtract the two inequalities. At first
\begin{eqnarray*}
\mu F_i(\bar{x},\frac{p_\e+2(\bar{x}-x_0)}{\mu})- F_i(\bar{y},p_\e)
&=& \mu F_i(\bar{x}, \frac{p_\e+2(\bar{x}-x_0)}{\mu})-  F_i(\bar{x},p_\e)\\
&& +  F_i(\bar{x},p_\e)- F_i(\bar{y},p_\e).
\end{eqnarray*}
By convexity of $F_i,$ we have
\begin{eqnarray*}
\mu F_i(\bar{x}, \frac{p_\e+2(\bar{x}-x_0)}{\mu})-  F_i(\bar{x},p_\e)
\geq -(1-\mu) F_i(\bar{x},\frac{2(\bar{x}-x_0)}{1-\mu})
\mathop{\longrightarrow}_{\e\to 0}  -(1-\mu) F_i(x_0,0)=0
\end{eqnarray*}
by \eqref{H1}.
On the other hand, using~\eqref{estim-class} and the uniform continuity of
$F_i$ on the compact subset $\T^N\times \overline{B}(0,C),$
where $C$ is given by~\eqref{Cgrad}, we have
\begin{eqnarray*}
|F_i(\bar{x},p_\e)- F_i(\bar{y},p_\e)|\leq o_\e(1)
\end{eqnarray*}
where $o_\e(1)\to 0$ as $\e\to 0.$
Moreover
\[
\sum_{j=1}^m d_{ij}(\bar{x})\mu u_j(\bar{x})-d_{ij}(\bar{y}) v_j(\bar{y}) =
\sum_{j=1}^m d_{ij}(\bar{x})(\mu u_j(\bar{x}) - v_j(\bar{y}))
 +  \sum_{j=1}^m (d_{ij}(\bar{x})-d_{ij}(\bar{y})) v_j(\bar{y})
\]
and
\begin{eqnarray*}
\sum_{j=1}^m (d_{ij}(\bar{x})-d_{ij}(\bar{y})) v_j(\bar{y})
\mathop{\longrightarrow}_{\e\to 0} 0
\end{eqnarray*}
since $v_i$ are bounded and $d_{ij}$ is continuous.
Finally, we obtain
\begin{eqnarray}\label{ineg-fondam}
\sum_{j=1}^m d_{ij}(\bar{x})(\mu u_j(\bar{x}) - v_j(\bar{y}))
\leq (\mu -1)f_i(\bar{x})+ o_\e(1).
\end{eqnarray}

Since $\mu u_j - v_j$ is $USC$, for all $j,$
\begin{eqnarray}\label{formabc}
\mathop{\rm lim\,sup}_{\e\to 0}(\mu u_j(\bar{x}) - v_j(\bar{y}))\leq
\mu u_j(x_0) - v_j(x_0)\leq M_\mu.
\end{eqnarray}
Recalling that $d_{ij}\leq 0$ for $j\not= i,$ it follows
\begin{eqnarray}\label{formabcd}
d_{ij}(\bar{x})  (\mu u_j(\bar{x}) - v_j(\bar{y}))\geq d_{ij}(\bar{x}) M_\mu + o_\e(1)
\quad \text{for } j\not= i.
\end{eqnarray}
Moreover, since $i\in \mathcal{I},$ we have $\mu u_i(\bar{x}) - v_i(\bar{y})\geq M_\mu$ and
\begin{eqnarray*}
d_{ii}(\bar{x})(\mu u_i(\bar{x}) - v_i(\bar{y}))\geq d_{ii}(\bar{x})M_\mu.
\end{eqnarray*}

From \eqref{ineg-fondam}, we get
\begin{eqnarray}\label{ineg745}
\left(\sum_{j=1}^m d_{ij}(\bar{x})\right)M_\mu\leq (\mu -1)f_i(\bar{x})+ o_\e(1)
\end{eqnarray}
which leads to a contradiction from $\e$ small enough since
$M_\mu> 0$ and, by \eqref{iiii}, either $\sum_{j=1}^m d_{ij}(x_0)> 0$ or  $f_i(x_0)>0.$

The proof of the theorem in the classical case reduces to
Case 2. Indeed, in the classical case, $\F=\emptyset$ and,
regardless $\mathcal{I}=\{1,\cdots m\}$ or not, we can always
choose $i\in\mathcal{I}$ in order that~\eqref{iiii} holds.
Notice that we do not need~\eqref{regf}. It suffices to send
$\epsilon\to 0$ and $\mu\to 1$ in~\eqref{ineg745}.
\medskip

\noindent{\it Case 3: $\mathcal{I}\not=\{1,\cdots m\}$.}
Using that $D(x_0)$ is irreducible, there exist $i\in \mathcal{I}$ and
$k\not\in \mathcal{I}$ such that $d_{ik}(x_0)<0.$
We argue as in Case 2 to obtain \eqref{ineg-fondam}.
Inequalities \eqref{formabc} and \eqref{formabcd} hold true in this case too.
But we need a more precise estimate for the index $k.$ Since
$k\not\in \mathcal{I},$
\begin{eqnarray}\label{formabce}
\mathop{\rm lim\,sup}_{\e\to 0}(\mu u_k(\bar{x}) - v_k(\bar{y}))\leq
\mu u_k(x_0) - v_k(x_0)\leq M_\mu-\eta
\quad \text{for some } \eta >0.
\end{eqnarray}
From \eqref{ineg-fondam}, \eqref{formabc}, \eqref{formabcd} and \eqref{formabce}, we obtain
\begin{eqnarray*}
d_{ii}(\bar{x})M_\mu + \sum_{j\not= i,k} d_{ij}(\bar{x})M_\mu
+ d_{ik}(\bar{x}) (M_\mu -\eta )\leq (\mu -1)f_i(\bar{x})+ o_\e(1).
\end{eqnarray*}
It follows
\begin{eqnarray*}
-d_{ik}(x_0)\eta \leq
\left(\sum_{j=1}^m d_{ij}(\bar{x})\right)M_\mu
-d_{ik}(x_0)\eta \leq (\mu -1)f_i(\bar{x})+o_\e(1) \leq o_\e(1)
\end{eqnarray*}
which leads to a contradiction for small $\e$ since $d_{ik}(x_0)<0.$
\medskip

\noindent{\it End of the proof.}
The only possible case is $M_\mu\leq (1-\mu)|v|_\infty$ which implies that
$M_1\leq 0.$ The proof is complete.
\end{proof}
\begin{rem} \label{rmq-comp}
The classical case \eqref{strict-pos} corresponds  to the scalar case $\l u+H(x,Du)=0$
with $\l>0$ and, in this case, we always have  existence and uniqueness
of the solution.
Theorem~\ref{confr} (ii) says that $\A$ is an uniqueness  
set for \eqref{statioHJ}.
We recall that for the single equation $F(x, Du)=f(x)$, where $F$, $f$
satisfy~\eqref{H0}--\eqref{regf}, the uniqueness set is $\{x\in\T^N:\, f(x)=0\}$
(see Fathi~\cite{fathi10}).
Notice that $\A$ maybe empty so we have automatically comparison
(but the existence of solutions may fail). When $\A$ is not empty,
it is enough to assume $u_i\leq v_i$ for {\it one} $i$ (this is a
consequence of the irreducibility of the coupling matrix).
\end{rem}

Before giving an existence result for \eqref{statioHJ},
we prove Lipschitz regularity of subsolutions.

\begin{lem}\label{subsol-lip}
Assume that \eqref{H3}
and~\eqref{H0}-\eqref{regf} hold.
Let $u\in USC( \T^N)$  be a bounded viscosity subsolution to \eqref{statioHJ}.
Then $u$ is Lipschitz continuous in $\T^N$ with a constant
$L=L(H_1,\cdots, H_m, D, |u|_\infty).$
If~\eqref{degenerate}-\eqref{Dcst-irr} hold with $E_1=E_2=\T^N,$
then $L$ is independent of $|u|_\infty.$
\end{lem}

\begin{proof}
From the coercitivity of the Hamiltonians $F_i$ and
\cite[Lemma 2.5 p.33]{barles94}, it is sufficient to prove
that $u_i$ is a viscosity subsolution of $F_i(x,Du_i)\le C$ in $\T^N.$
We have
\begin{eqnarray*}
F_i(x,Du_i)\leq \sum_{j=1}^m |d_{ij}(x)||u_j(x)| + f_i(x),
\end{eqnarray*}
which gives the result with a Lipschitz constant depending
on $H_1,\cdots, H_m, |D|_\infty, |u|_\infty.$

Now, if~\eqref{degenerate}-\eqref{Dcst-irr} hold in $\T^N,$ then, from
Lemma~\ref{lem-rang} (i), there exists
a continuous function $\Lambda:\T^N\to\R^m,$ $\Lambda >0,$ such that
$D(x)^T\Lambda(x)= 0$ for all $x\in\T^N.$
By multiplying Equations~\eqref{HJLm} by $\Lambda_i(x)$
and summing for $i=1,\dots, m,$
we obtain
\begin{eqnarray}\label{somequ123}
&&\sum_{i=1}^m \Lambda_i(x)F_i(x,Du_i)
+ \sum_{i,j=1}^m\Lambda_i(x)d_{ij}(x)u_j=\sum_{i=1}^m \Lambda_i(x)f_i(x).
\end{eqnarray}
We have
\begin{eqnarray*}
\sum_{i,j=1}^m\Lambda_i(x)d_{ij}(x)u_j
= \sum_{j=1}^m\left(\sum_{i=1}^m\Lambda_i(x)d_{ij}(x)\right)u_j
=0
\end{eqnarray*}
since $D(x)^T\Lambda(x)= 0$.
It follows from~\eqref{somequ123} that
\begin{eqnarray*}
 \sum_{i=1}^m \Lambda_i(x)F_i(x,Du_i)\leq \sum_{i=1}^m |\Lambda_if_i|_\infty,
\end{eqnarray*}
By the compactness of
$\T^N$ and the continuity of $\Lambda,$ there
exists $\eta=\eta (D) >0$ such that $\Lambda_i\geq \eta$
on $\T^N$ for all $i.$ It completes the proof.
\end{proof}

To state an existence result with prescribed values on
$\mathcal{A},$ we need to introduce some definitions of
Fathi and Siconolfi~\cite{fs05}.
Define, for every $x\in\T^N,$ $p\in\R^N,$
\begin{eqnarray}\label{Fetf}
F(x,p)= \mathop{\rm max}_{1\leq i\leq m} F_i(x,p)
\quad {\rm and} \quad
f(x)= \mathop{\rm min}_{1\leq i\leq m} f_i(x)
\end{eqnarray}
and set
\begin{eqnarray*}
S(y,x)={\rm max}\{ u(x)\; : \;
u \text{ subsolution of } F(x,Du)\leq f(x) \ {\rm on} \ \T^N
\text{ with } u(y)=0\}.
\end{eqnarray*}

\begin{prop}\label{subsol}
Assume  \eqref{H3}, \eqref{Dcst-irr} with $E_2=\T^N,$ \eqref{H0}--\eqref{regf}.
\begin{itemize}

\item[(i)] (Classical case) If~\eqref{strict-pos} holds, then there
exists a unique continuous viscosity solution
of~\eqref{statioHJ}.

\item[(ii)] (Degenerate case) Suppose that~\eqref{degenerate} holds
with $E_1=\T^N$
and $\F\not=\emptyset.$
For any continuous function $g:\F\to\R$ satisfying
\begin{eqnarray}\label{cond-fathi-sico}
g(x)\left(\sum_{i=1}^m \Lambda_i(x)\right)^{-1}
- g(y)\left(\sum_{i=1}^m \Lambda_i(y)\right)^{-1}
\leq S(y,x),
\end{eqnarray}
there exists a unique continuous viscosity solution $u$
of~\eqref{statioHJ} such that
\begin{eqnarray}\label{cond-sur-F}
\sum_{i=1}^m \Lambda_i(x)u_i(x)= g(x)\qquad x\in\F,
\end{eqnarray}
where the continuous function $\Lambda:\T^N\to\R^m,$ $\Lambda >0,$
is given by Lemma~\ref{lem-rang}.

\end{itemize}
\end{prop}

\begin{rem}
The assumption~\eqref{cond-fathi-sico} can be seen as a compatibility
condition.
We cannot prescribe any function $g$ on~$\mathcal{A}.$
Indeed, for instance, if $\Lambda_i=1$ for all $i$ and 
$g$ has a large Lipschitz
constant compared to the one given by Lemma~\ref{subsol-lip}, then
it is straightforward to see that it is not possible
to build a solution $u$ satisfying~\eqref{cond-sur-F}.
\end{rem}

\begin{proof}
The proof of the existence of a solution is based on Perron's method.
We start by building subsolutions and supersolutions.

From \eqref{H1}, \eqref{regf} and \eqref{H3}, we obtain that $\psi_C=(C,\dots,C)$
is a subsolution of \eqref{statioHJ}
for every nonpositive constant $C\leq 0.$ This subsolution is
suitable in the case (i).

In the case (ii), we need to build a subsolution which
satisfies~\eqref{cond-sur-F}, which is more tricky.
Note that $F$ and $f$ given by~\eqref{Fetf} still satisfy~\eqref{H0}--\eqref{regf}.
Since the compatibility condition~\eqref{cond-fathi-sico} holds, we
can use the result of~\cite[Prop. 4.7]{fs05}:
there exists a subsolution $\psi$ of
\begin{eqnarray*}
F(x,D\psi)=f(x), \quad x\in\T^N
\qquad {\rm with } \ \psi(x)=\frac{g(x)}{\sum_{i=1}^m \Lambda_i(x)}
\quad {\rm on} \ \F.
\end{eqnarray*}
Then, the following computation shows that $\underline\Psi:=(\psi,\cdots,\psi)$
is a subsolution of~\eqref{statioHJ} such that~\eqref{cond-sur-F}
holds: recalling~\eqref{degenerate}, for all $x\in\T^N,$
\begin{eqnarray*}
F_i(x,D\underline\Psi_i)+\sum_{j=1}^m d_{ij}\underline\Psi_j
= F_i(x,D\psi)+\psi(x)\sum_{j=1}^m d_{ij}
\leq F(x,D\psi)\leq f(x)\leq f_i(x).
\end{eqnarray*}

Now we turn to the construction of a supersolution.
Under assumption~\eqref{strict-pos}, we can use $\psi_C$
again with a suitable choice of $C\geq 0$ since, for all $i,$
\begin{eqnarray*}
F_i(x,D(\psi_C)_i(x))+
\sum_{j=1}^m d_{ij}(x)(\psi_C)_i(x)
= F_i(x,0) + \left(\sum_{j=1}^m d_{ij}(x)\right)C
\geq \eta C
\end{eqnarray*}
where
\begin{eqnarray*}
\eta =\mathop{\rm inf}_{\T^N} \sum_{j=1}^m d_{ij} >0
\end{eqnarray*}
by \eqref{strict-pos}. It then suffices to choose
\begin{eqnarray}\label{choix-c11}
C\geq \eta^{-1}\, \mathop{\rm max}_{1\leq i\leq m}|f_i|_\infty.
\end{eqnarray}

In the second case, since $\F\not=\emptyset,$
we may define $d_\F (x)={\rm dist}(x,\F).$
Choose $C$ as in \eqref{choix-c11} with $\eta=1.$
By coercivity
of $F_i$ (see \eqref{H1}), there exists $C' >0$ such that, for all
$x\in\T^N, p\in\R^N,$ if  $|p|\geq C',$ then
$F_i(x,p)\geq C$ for all $i.$
We claim that $\overline{\Psi}=(\overline{\psi},\dots,\overline{\psi})$
with $\overline{\psi}=C' d_\F+C,$ is a
viscosity supersolution of the stationary problem.
Indeed, let $i\in\{1,\cdots ,m\}$ and let $\varphi$ be a $C^1$ function such that
$\overline{\psi}-\varphi$ achieves a minimum at $x_0\not\in \F.$
In $\T^N\setminus\F,$ it is well-known that $|Dd_\F|=1$ in the viscosity sense.
Since $d_\F+C/C'-\varphi/C'$ achieves a minimum
at $x_0$, we get $|D\varphi(x_0)|\geq C'.$ It follows that
\begin{eqnarray*}
F_i(x_0, D\varphi(x_0))+\left(\sum_{j=1}^m d_{ij}(x_0)\right)\overline{\psi}(x_0)
\geq C\geq f_i(x_0)
\end{eqnarray*}
since both $\sum_{j=1}^m d_{ij}(x_0)$ and $\overline{\psi}(x_0)$ are
nonnegative.
If $x_0\in\F,$ then  $f_i(x_0)$
vanishes.
Since $F_i\geq 0,$ the supersolution inequality obviously holds on $\F.$
The claim is proved.

Then, we apply the extension of Perron's method
to systems, see~\cite{el91, ishii92}.
Using the comparison principle~\ref{confr} and
following readily the proof
of~\cite[Prop. 2.1]{el91}, we obtain that the supremum of subsolutions
which are less than $\psi_C$ (resp. $\overline\Psi$) is a solution
in the case (i) (resp. (ii)).
From Lemma~\ref{subsol-lip}, the subsolutions of~\eqref{statioHJ}
are Lipschitz continuous with a constant $L$ depending
only on $H_1,\cdots, H_m$ and $D.$ It follows that the
supremum is still Lipschitz continuous.
In the case of (ii), note that the supremum still
satisfies~\eqref{cond-sur-F}.
\end{proof}
\begin{rem} \label{Existence}
Under the assumptions~\eqref{H3}, \eqref{H0}--\eqref{regf}  there always
exists a subsolution. The assumption $\F\not=\emptyset$ is needed to build a supersolution
when for instance $\D_i=\T^N$ for all $i.$ Notice that, in the case (i) of 
Proposition~\ref{subsol}, we do not need to assume~\eqref{regf}, the $f_i$'s may be
any continuous functions in $\T^N.$
\end{rem}

\section{The ergodic problem}\label{s4}

In this section, we study the solutions of \eqref{statioHJ} with
an ergodic constant, that is:
\begin{equation}\label{HJ}
    H_i(x,Dv_i)+\sum_{j=1}^{m}d_{ij}(x)v_j=c_i\qquad x\in\T^N, \ 1\leq i\leq m,
\end{equation}
where $H_i$ is given by \eqref{Hroquejo}.
When~\eqref{strict-pos} holds, then, see Remark \ref{Existence},
for any $c=(c_1,\dots,c_m)$ there is
a unique viscosity solution $v\in C(\T^N)$. Hence we concentrate on the case  $\mathcal{D}_i$   not empty
for some $i $ and we   consider the ergodic approximation to \eqref{HJ}:
for  $\l \in (0,1)$, let  $v^\l=(v_1^\l,\dots,v^\l_m)$ be the solution  of
\begin{equation}\label{HJLm}
 \l v_i+ H_i(x,Dv_i)+\sum_{j=1}^{m}d_{ij}(x)v_j=0
\qquad x\in\T^N, \ 1\leq i\leq m.
\end{equation}

\begin{lem}\label{reg-vl}
Assume~\eqref{H3}, \eqref{Dcst-irr} with $E_2=\T^N$ and \eqref{H0}--\eqref{regf}.
Then there exist a unique viscosity solution $v^\l$ of~\eqref{HJLm}
and some constants $C_0,M>0$ independent of $\l$ such that
$v^\l$ is Lipschitz continuous with constant $C_0$ and
\begin{eqnarray}\label{estim-vl}
0\leq v_i^\l\leq \frac{M}{\l}
\quad {and} \quad
\left|\sum_{j=1}^m d_{ij}v_j^\l\right|\leq M,
\qquad i=1,\dots,m.
\end{eqnarray}
\end{lem}
\begin{proof}
We first observe that the system \eqref{HJLm}
satisfies the assumption \eqref{strict-pos} for all $\l>0$ so
Theorem \ref{confr} and Proposition \ref{subsol} (classical case)
hold. Hence there exists a unique solution $v^\l.$
Moreover, since $\ov u=(M/\l,\dots,M/\l)$, $\underline u=(-M/\l,\dots,-M/\l)$, where
\begin{eqnarray*}
M\ge \sup_{1\leq i\leq m} \sup_{x\in\T^N}|F_i(x,0)| +|f_i(x)|,
\end{eqnarray*}
are, respectively, a super and a subsolution of \eqref{HJLm}, we have
\begin{equation}\label{E1m}
   -\frac{M}{\l}\le v^\l_i\le \frac{M}{\l} \qquad i=1,\dots,m.
\end{equation}
In fact, in our case where $H_i(x,p)=F_i(x,p)-f_i(x)$ with $f_i\geq 0$
and $F_i(x,0)=0,$ we have that $(0,\dots,0)$ is a subsolution
and therefore we obtain the more precise estimate
\begin{equation}\label{vlpos}
 0\le v^\l_i \qquad i=1,\dots,m.
\end{equation}

We want to prove that the $v^\l_i$'s are Lipschitz continuous uniformly with
respect to $\l.$ We argue as in the second part of the proof
of Lemma~\ref{subsol-lip}, using now Lemma~\ref{lem-rang} (ii)
with $E_1=\emptyset$ and $E_2=\T^N,$
since~\eqref{degenerate} is not assumed here.
Equality~\eqref{somequ123} is replaced with
\begin{eqnarray*}
&&\l \sum_{i=1}^m  \Lambda_i(x)v^\l_i + \sum_{i=1}^m \Lambda_i(x)F_i(x,Dv^\l_i)
+ \sum_{i,j=1}^m\Lambda_i(x)d_{ij}(x)v^\l_j=\sum_{i=1}^m \Lambda_i(x)f_i(x)
\end{eqnarray*}
and
\begin{eqnarray*}
\l \sum_{i=1}^m  \Lambda_i(x)v^\l_i
+ \sum_{j=1}^m\left(\sum_{i=1}^m\Lambda_i(x)d_{ij}(x)\right)v^\l_j\geq 0
\end{eqnarray*}
since $D(x)^T\Lambda(x)\geq 0$ and~\eqref{vlpos} holds.
We then conclude as in Lemma~\ref{subsol-lip} that $v^\l$ is Lipschitz
continuous with some constant $C_0=C_0(H_1,\dots, H_m, D).$

Finally, since for all $i,$ $\l v_i^\l$ and $H_i(x,Dv_i^\l)$ are bounded
independently of $\l$ in~\eqref{HJLm}, it is true also for $\sum_{j=1}^m d_{ij}v_j^\l$.
\end{proof}
The next theorem gives a first set of assumptions under
which we may solve \eqref{HJ}.

\begin{thm}\label{thm-ergodic}
We assume \eqref{H3}, \eqref{Dcst-irr} with $E_2=\T^N$, \eqref{H0}--\eqref{regf}   and
\begin{equation}
 \F=\{x\in \T^N:\sum_{i=1}^m f_i(x)=0\}\not= \emptyset.\label{Fnonvide}
\end{equation}
Let $x^*\in \F$. If $v^\l$ is the solution of \eqref{HJLm}, then, up to extract some subsequence as $\l\to 0,$
\begin{eqnarray*}
& & -\l v^\l\to c=(c_1,\cdots, c_m)\in\R^m,\\
& & v^\l-v^\l(x^*)\to v=(v_1,\dots,v_m) \ { in} \ C(\T^N)
\end{eqnarray*}
and $(c,v)\in\R^m\times C(\T^N)$ is solution to~\eqref{HJ}
with $v$ Lipschitz continuous and $c\in {\rm ker}\,D(x)$ for
all $x.$
Moreover
\begin{eqnarray}
&& v_i(x)=0 \quad
\text{for all } x\in\F,\ 1\leq i\leq m,\label{condV}\\
&& c_i=0\quad \text{for all }1\leq i\leq m,
\end{eqnarray}
and $c=(0,\dots ,0)$ is the unique constant vector in ${\rm ker}\, D(x),$
for all $x\in\T^N,$ such that~\eqref{HJ} has a solution.
\end{thm}

\begin{proof}
Fix $x^*\in \T^N$ and let
$w^\l(x)=v^\l(x)-v^\l(x^*).$ From Lemma~\ref{reg-vl},
$w^\l$ is Lipschitz continuous and bounded since
$\T^N$ is bounded. From Ascoli's theorem, up to subsequences,
there exist a constant $c\in\R^m$ and a Lipschitz continuous function $v$
such that
\begin{eqnarray}\label{convl}
-\l v^\l(x^*)\to c \quad {\rm and} \quad w^\l\to v \ {\rm in} \ C(\T^N)
\quad {\rm as} \ \l\to 0.
\end{eqnarray}
Notice that $v$ depends on $x^*$ but not $c$ since, for any $x^*,y^*\in\T^N,$
\begin{eqnarray}\label{c-indep-x}
|-\l v^\l(x^*)+\l v^\l(y^*)|\leq \l C_0 |x^*-y^*|\to 0 \quad {\rm as} \ \l\to 0.
\end{eqnarray}
Moreover, multiplying~\eqref{HJLm} by $\l$ for all $i$ and sending $\l\to 0,$
we obtain $-\sum_j d_{ij}(x)c_i=0$ whichs gives $D(x)c=0$ and therefore
$c\in {\rm ker}\,D(x).$

Let $x\in \F.$ Since $F_i\geq 0$ and $f_i(x)=0,$
we observe that~\eqref{HJLm} implies
\begin{equation}\label{E5m}
    \l v^\l_i(x)+\sum_{j=1}^m d_{ij}(x) v^\l_j(x)\le 0 \qquad i=1,\dots,m.
\end{equation}
By Lemma~\ref{lem-rang} (ii),
there exists a continuous
$\Lambda:\T^N\to\R^m,$ $\Lambda >0$ such that $D(x)^T\Lambda(x)\geq 0$
on $\T^N$ and, since $v^\l_i(x)\geq 0$ by~\eqref{vlpos}, we obtain
\begin{eqnarray*}
\l \sum_{i=1}^m \Lambda_i(x)v_i^\l(x)
\leq\l \sum_{i=1}^m \Lambda_i(x)v_i^\l(x)
+\sum_{j=1}^m v_j^\l(x)\left( \sum_{i=1}^m \Lambda_i(x)d_{ij}(x)\right)
\leq 0.
\end{eqnarray*}
It follows
\begin{eqnarray}\label{vl0F}
v^\l_i=0 \quad {\rm on \ } \F, \quad  i=1,\dots,m.
\end{eqnarray}

From now on, we choose $x^*\in \F.$ From~\eqref{vl0F}, we obtain
that $c=(0,\dots ,0)$ and $v$ vanishes on $\F.$
Observe that $w^\l$ is a solution of the system
\begin{eqnarray*}
\l w^\l_i+  H_i(x,Dw^\l_i)+\sum_{j=1}^m d_{ij}(x)w^\l_j(x)
+ \sum_{j=1}^m d_{ij}(x)v^\l_j(x^*)=0
\quad {\rm in} \ \T^N, \ 1\leq i\leq m,
\end{eqnarray*}
and that the last term in the left-hand side is 0 because
of~\eqref{vl0F}. By~\eqref{convl} and the stability for viscosity solutions
as $\l\to 0,$ we conclude that the couple
$((0,\dots,0),v)$ is a solution of the system \eqref{HJ}.

Suppose that $(c,v)$ and $(\tilde c, \tilde v)$ are two solutions
of~\eqref{HJ} with $c,\tilde c\in  {\rm ker}\, D(x)$ for all $x.$
Define $w(x,t)=v(x)-ct$ and $\tilde w(x,t)=\tilde v(x)-\tilde ct.$
Since $c,\tilde c\in  {\rm ker}\, D(x),$ we have that
$w$ and $\tilde w$ are solutions of~\eqref{HJE} with initial conditions
$v$ and $\tilde v$ respectively. By comparison (see Proposition~\ref{PrComp}),
we get that, for all $x\in\T^N, t\geq 0$ and $1\leq i\leq m,$
\begin{eqnarray*}
v_i(x)-\tilde{v}_i(x) + ({\tilde c}_i-c_i)t \leq
\mathop{\rm max}_{1\leq j\leq m} \mathop{\rm sup}_{\T^N}\, (v_j-\tilde{v}_j)^+.
\end{eqnarray*}
Therefore ${\tilde c}_i\leq c_i.$ Exchanging the role
of $v$ and $\tilde v,$ we obtain that $c=\tilde c.$
\end{proof}

We give another case
where we can solve the ergodic problem~\eqref{HJ}.
In particular, note that
$\mathcal{F}$ may be empty.

\begin{thm}\label{thm-ergodic-bis}
We assume that 
\begin{eqnarray}
&& D \text{ does not depend on } x\label{Dcste}
\end{eqnarray}
and that \eqref{H3}, \eqref{degenerate}, \eqref{Dcst-irr}
and \eqref{H0}--\eqref{regf} hold.
Then, there is at least one solution $(c,v)\in\R^m\times C(\T^N)$ to~\eqref{HJ}
with $v$ Lipschitz continuous, $c=(c_1,\dots,c_1)$ is unique
in ${\rm ker}\,D$ and
\begin{eqnarray}\label{caract-c}
\sum_{i=1}^m \Lambda_i \mathop{\rm min}_{\T^N}f_i(x)
\leq -c_1 \sum_{i=1}^m \Lambda_i \leq
\mathop{\rm min}_{\T^N} \sum_{i=1}^m \Lambda_i f_i(x),
\end{eqnarray}
where the constant vector $\Lambda >0$ is given by Lemma~\ref{lem-rang} (i).
\end{thm}

\begin{proof}
Let $v^\l$ be the solution  of~\eqref{HJLm} and fix any
$x^*\in\T^N.$ We argue as in the proof of Theorem~\ref{thm-ergodic}.
From Lemma~\ref{reg-vl} and Ascoli's theorem, there exist
$c=(c_1,\dots,c_m)\in\R^m$ and $v\in C(\T^N)$
such that, up to extract subsequences, for $i=1,\dots,m,$
\begin{eqnarray*}
\l v_i^\l(x^*)\to -c_i
\quad {\rm and} \quad
v^\l_i-v^\l_i(x^*)\to v_i \ {\rm in} \ C(\T^N),
\quad \ {\rm as} \ \l\to 0.
\end{eqnarray*}
Notice that $c_i$ does not depend on the choice
of $x^*$, (see~\eqref{c-indep-x}), and $c\in {\rm ker}\,D.$
We write
\begin{eqnarray}\label{eq432}
&&\l v^\l_i+  H_i(x,D(v^\l_i-v^\l_i(x^*)))+\sum_{j=1}^m d_{ij}(v^\l_j-v^\l_i(x^*))
+ \sum_{j=1}^m d_{ij}v^\l_j(x^*)=0
\quad {\rm in} \ \T^N,
\end{eqnarray}
From~\eqref{estim-vl}, we obtain that some subsequences of both
$\l v_i^\l(x^*)$ and $\sum_{j=1}^m d_{ij}v_j^\l (x^*)$ converge. We call the
second limit $\rho_i=\rho_i(x^*)$.
Using the positive vector $\Lambda$ given by Lemma~\ref{lem-rang} (i),
we have
\begin{eqnarray*}
\sum_{i=1}^m\Lambda_i \sum_{j=1}^m d_{ij}v^\l_j(x^*)
= \sum_{j=1}^m\left(\sum_{i=1}^m\Lambda_i d_{ij}\right)v^\l_j(x^*)
=0.
\end{eqnarray*}
Passing to the limit in the above formula, 
it follows $\langle \Lambda, \rho\rangle =0,$
where $\langle\cdot,\cdot\rangle$
is the standard inner product. Since $D^T \Lambda =0$ and the rank of $D$
is $m-1,$ we get that the image ${\rm im}\,D$ of $D$ is $\Lambda^\perp.$ Thus
$\rho\in {\rm im}\,D$ and
there exists $\tilde{\rho}\in\R^m$ such that
$D\tilde{\rho}=\rho.$ 
We then use the stability result for viscosity solutions and
pass to the limit in~\eqref{eq432}. We get
\begin{eqnarray}
H_i(x,Dv_i)+\sum_{j=1}^m d_{ij}v_j(x)
+ \sum_{j=1}^m d_{ij}\tilde{\rho}_j=c_i
\quad {\rm in} \ \T^N, \ 1\leq i\leq m.
\end{eqnarray}
Then $(c,v(\cdot)+\tilde{\rho})$ is solution to~\eqref{HJ}. The function
$\tilde{v}=v+\tilde{\rho}$ depends on $x^*$ but not $c.$

From Lemma~\ref{lem-rang}, the kernel of $D$ is the line spanned by $(1,\dots,1).$
Thus, any $c\in {\rm ker}\,D$ has the form $(c_1,\dots,c_1).$
The proof of uniqueness of $c$ is the same as the one in Theorem~\ref{thm-ergodic}.

It remains to prove~\eqref{caract-c}.
We use again the vector $\Lambda$ given by
Lemma~\ref{lem-rang} (i).
On the one hand, multiplying~\eqref{HJ} by $\Lambda_i$
and summing them for $i=1,\dots,m,$
we obtain
\begin{eqnarray}\label{1ineg}
0\leq \sum_{i=1}^m \Lambda_i f_i(x)+ c_1\sum_{i=1}^m \Lambda_i,
\qquad x\in\T^N.
\end{eqnarray}
On the other hand, let $x_i\in\T^N$ be a minimum of the continuous
function $u_i$ and set $\bar{u}_i= u_i(x_i).$
 At $x_i,$ the equation~\eqref{HJ} reads
\begin{eqnarray}\label{eq987}
 \sum_{j=1}^m d_{ij}\bar{u}_j \geq \sum_{j=1}^m d_{ij}u_j(x_i)= f_i(x_i)+c_1,
\end{eqnarray}
since $\bar{u}_j\leq u_j(x_i)$ and $d_{ij}\leq 0$ for $i\not=j.$
Multiplying~\eqref{eq987}  by $\Lambda_i$ and summing them for $i=1,\dots,m,$
we get
\begin{eqnarray*}
0=\sum_{j=1}^m \bar{u}_j \sum_{i=1}^m \Lambda_i d_{ij} \geq
\sum_{i=1}^m \Lambda_i(f_i(x_i)+c_1)
\geq \sum_{i=1}^m \Lambda_i(\mathop{\rm min}_{\T^N}f_i+c_1).
\end{eqnarray*}
Combining the previous inequality with~\eqref{1ineg},
we finally obtain~\eqref{caract-c}.
\end{proof}

\begin{rem} \
\begin{enumerate}
\item The inequality~\eqref{caract-c} gives a characterization
of $c$ when all the $f_i$'s achieve the same minimum
at the same point.
\item It would be interesting to prove Theorem~\ref{thm-ergodic-bis}
when the $d_{ij}$'s depend on $x.$ The difficulty is that the $\rho_i$'s
in the proof are now functions on $x$ and we do not obtain anymore
a solution of~\eqref{HJ}.
\end{enumerate}
\end{rem}

\begin{ex}
Consider the system
\begin{equation}\label{cplusfnegatif}
\left\{  \begin{array}{lr}
   |Dv^1|+v^1-v^2=f_1+c_1&\text{in $\T^N$},\\
   |Dv^2|-v^1+v^2=f_2+c_2&\text{in $\T^N$},
     \end{array}\right.
\end{equation}
with $f_1\equiv a, f_2\equiv b,$ $a,b>0.$
This system satisfies all the assumptions of Theorem~\ref{thm-ergodic}
with the exception of~\eqref{Fnonvide} since $f_1+f_2\equiv a+b>0.$
However, all the assumptions of Theorem \ref{thm-ergodic-bis}
hold and therefore we can solve~\eqref{cplusfnegatif}.
For instance, we can exhibit constant solutions $(c,v)$ with $c_1=c_2.$
An easy computation gives $c_1=-(a+b)/2$ and all constants
$v=(v_1,v_2)$ satisfying $v_1-v_2=(a-b)/2$ are suitable.
Notice that we can solve~\eqref{cplusfnegatif} even in some cases
where Theorem~\ref{confr} and Proposition~\ref{subsol} do not apply.
Indeed,~\eqref{cplusfnegatif} corresponds to the degenerate
case of  Theorem~\ref{confr} and
either $f_1+c_1$ or $f_2+c_1$ is negative
and~\eqref{regf} does not hold.
\end{ex}

\section{Large time behavior}\label{s5}

We are interested in the long-time behavior of the evolutive system
\begin{equation}\label{HJE}
\left\{
  \begin{array}{ll}\displaystyle
\frac{\partial u_i}{\partial t}+  H_i(x, Du_i )+\sum_{j=1}^{m}d_{ij}(x)u_j=0
   & (x,t)\in  \T^N\times (0,+\infty), \\[5pt]
  u_{i}(x,0)=u_{0,i}(x)&x\in \T^N,\ \ i=1, \cdots m,
  \end{array}
\right.
\end{equation}
where $H_i$ is of the form \eqref{Hroquejo} and
\begin{equation}
   \text{$u_{0,i}$ is continuous and $1$-periodic.}\label{HID}
\end{equation}
We start  giving some auxiliary results for the evolutive problem.
The following  two propositions come from~\cite{cll10}, where homogenization
of a   general class of monotone systems
which includes in particular the weakly coupled system~\eqref{HJE}, is studied.
Let us mention that Proposition~\ref{PrComp} is established in~\cite{cll10} 
under the additional assumption
\begin{eqnarray*}
&& \text{There exists a modulus of continuity $\omega$ such that,}\\
&& |F_i(x,p)-F_i(y,p)|\leq \omega((1+|p|)|x-y|)\quad
\text{for all $x,y\in\T^N,$ $p\in\R^N,$ $i=1,\dots,m,$}
\end{eqnarray*}
but a careful examination of the proof shows that we do not need it.
The coercivity of the $F_i$'s is enough (see the proof of Theorem~\ref{confr}).

\begin{prop}\label{PrComp}
Assume~\eqref{H3} and~\eqref{Hroquejo}--\eqref{regf}.
\begin{itemize}
  \item [(i)] If $u_0$, $v_0$ are two initial datas satisfying \eqref{HID} and $u$, $v$ are
respectively a viscosity subsolution and a supersolution of~\eqref{HJE}, then
for any $t\geq 0,$
\begin{eqnarray}\label{comp-max}
\mathop{\rm max}_{1\leq i\leq m}\mathop{\rm sup}_{\T^N}\, (u_i(\cdot,t)-v_i(\cdot,t))
\leq \mathop{\rm max}_{1\leq i\leq m}\mathop{\rm sup}_{\T^N}
\left(u_{i} (\cdot,0)-v_{i} (\cdot ,0)\right)^+.
\end{eqnarray}
  \item [(ii)]For any $u_0$ satisfying \eqref{HID}, there exists a unique continuous viscosity solution $u$
of \eqref{HJE}.
\end{itemize}
\end{prop}
A crucial step in the study of the large-time behavior of equations or systems
is to obtain compactness properties of the sequence
$(u(\cdot,t))_{t\geq 0}.$ As for a single Hamilton-Jacobi equation,
it relies on the coercitivity
of the Hamiltonians.
\begin{prop}\label{PrReg}
Under the assumptions of Theorem~\ref{thm-ergodic} or
Theorem~\ref{thm-ergodic-bis},
let $u_0\in W^{1,\infty}(\T^N)$ and $u$ be the solution of \eqref{HJE} with
initial datum $u_0$. Then
\begin{align}
   &|u(x,t)+c t|\le C & x \in\T^N,\,  t\in [0,\infty),\label{reg1}\\
   &|u(x,t)-u(y,s)|\leq L(|x-y|+|t-s|)
& x,y\in\T^N,\,  t,s\in [0,\infty),\label{reg2}
\end{align}
with  $C$, $L$ independent of time,
where $c$ is the ergodic constant given in Theorem~\ref{thm-ergodic}
or~\ref{thm-ergodic-bis}.
It follows that
\begin{eqnarray*}
\frac{u_i(x,t)}{t}\to -c_i \qquad \text{\it uniformly in} \ \T^N  \
{as} \ t\to +\infty, \quad i=1,\dots, m.
\end{eqnarray*}
\end{prop}

\begin{rem} \label{rem-borne}
Given $u_0$ satisfying~\eqref{HID}, set  $S(t)u_0=u(x,t)$ for $t\geq 0$
where $u$ is the solution of \eqref{HJE} with initial
datum $u_0$. Then  it is easy to see that $S(\cdot)$ generates a nonlinear, 
monotone,  nonexpansive semigroup in
$C(\T^N;\R^m)$.
Under the assumptions of Theorem~\ref{thm-ergodic},
$c=(0,\dots, 0),$ so $u$ is in $L^\infty(\T^N\times [0,+\infty)).$
It follows that $\{S(t)u_0, \, t\geq 0\}$ is relatively
compact in $C(\T^N).$ Therefore, the $\o$-limit set of an initial datum $u_0$ with
respect to the semigroup $S(t)$,
\begin{equation} \label{omega-lim-set}
    \o(u_0)=\{\psi=(\psi_1,\dots,\psi_m)\in C(\T^N):\,
\exists t_n\to\infty \text{ such that } \lim_{n\to\infty}S(t_n)u_0=\psi\}.
\end{equation}
is nonempty.
\end{rem}

\begin{proof}
The proof is based on the existence of a solution
to the ergodic problem which is used to estimate $u.$
It is classical but we provide the proof for reader's convenience.

Let $(c,v)$ be the solution of~\eqref{HJ}
given by Theorem~\ref{thm-ergodic} or~\ref{thm-ergodic-bis}. Since $c\in {\rm ker} D,$
$w(x,t)=v(x)-ct-|u_0|_\infty-|v|_\infty$ and $\tilde w(x,t)=v(x)-ct+|u_0|_\infty+|v|_\infty$
are respectively viscosity subsolution and supersolution to~\eqref{HJE}.
By comparison~\eqref{comp-max}, it follows that
\begin{equation}\label{UB}
    v(x)-|u_0|_\infty-|v|_\infty \le u(x,t)+ct 
\le v(x)+ |u_0|_\infty+|v|_\infty
\qquad (x,t)\in\T^N\times[0,\infty),
\end{equation}
which proves~\eqref{reg1}.
If we define
\begin{eqnarray*}
C:=\sup\left\{ \left|H_i\left(x,p\right)+\sum_j d_{ij}r_j\right|: \, x\in\R^N,\,
|r|\le | u_0|_\infty,\,  |p|\leq |Du_0|_\infty, \,1\leq i\leq m \right\},
\end{eqnarray*}
it is easy to see that $v^\pm (x,t)=( u_{0,1}(x) \pm Ct,\cdots , u_{0,m}(x) \pm Ct)$
are viscosity subsolution and supersolution of~\eqref{HJE}.
By Proposition \ref{PrComp}, it follows
\begin{eqnarray}\label{comp999}
v^-\leq u \leq v^+ \quad {\rm in} \ \T^N\times [0,+\infty).
\end{eqnarray}
Let $h\geq 0$ and
note that, since the $H_i$'s are independent of $t,$
$u (\cdot ,\cdot +h)$ is still a solution of \eqref{HJE}
with initial data $u (\cdot ,h).$
By \eqref{comp-max} and \eqref{comp999}, we get
for all $i=1,\dots ,m,$ $(x,t)\in \R^N\times [0,+\infty),$
\begin{eqnarray*}
 u_i (x,t+h)-u_i (x,t)
\leq \mathop{\rm max}_{1\leq j\leq m}\mathop{\rm sup}_{\T^N}
\left(u_{j}(\cdot,h)-u_{0,j}\right)^+
\leq Ch,
\end{eqnarray*}
and therefore $u_i$ is Lipschitz continuous with respect to $t$ for every $x$ with
\begin{eqnarray*}
\left|\frac{\partial u_i}{\partial t}\right|_\infty \leq C.
\end{eqnarray*}
with $C$ is independent of $t$.
From   \eqref{HJE} and \eqref{UB}, we obtain, in the viscosity sense,
\begin{eqnarray*}
F_i(x,Du_i) \leq C' \qquad(x,t)\in\T^N\times [0,+\infty).
\end{eqnarray*}
It follows from Lemma~\ref{subsol-lip}
that $u_i$ is Lipschitz continuous in $x$ for every $t$ with
$|Du_i|_\infty \leq L_i$ (with $L_i$ independent of $t$).
\end{proof}

We now state and prove our convergence result in the case of systems,
under the assumptions of Theorem~\ref{thm-ergodic}.

\begin{thm}\label{LTB}
Assume~\eqref{H3}, \eqref{Dcst-irr} with $E_2=\T^N,$
and \eqref{H0}--\eqref{regf}.
Suppose that $\mathcal{A}\not=\emptyset.$
For  every $u_0$ satisfying~\eqref{HID},
there exists a solution $u_\infty$ to \eqref{statioHJ} such that
the solution $u$ of \eqref{HJE} with initial datum $u_0$ satisfies
\[
\lim_{t\to\infty}|u(\cdot,t)-u_\infty|_\infty=0.
\]
Moreover
$(u_\infty)_i =(u_\infty)_j$ on $\A$ for all $i,j.$
\end{thm}

\begin{proof} The proof is divided in several steps. \\
\noindent{\it Step 1. Reduction to Lipschitz initial datas.}
Given $u_0$ satisfying~\eqref{HID}, set  $S(t)u_0=u(x,t).$  
Since the semigroup $S(t)$ is nonexpansive, see Remark~\ref{rem-borne},
it is sufficient to show the result for $u_0\in W^{1,\infty}(\T^N)$.
\smallskip

\noindent{\it Step 2. A positive linear combination of the $u_i$'s
is nonincreasing on $\mathcal{A}.$}
Since~\eqref{degenerate} holds with $E_1=\mathcal{A},$
from~\eqref{Dcst-irr} and Lemma~\ref{lem-rang},
there exists a positive continuous function
$\Lambda=(\Lambda_1,\dots,\Lambda_m):\mathcal{A}\to \R^m$
such that $D(x)^T\Lambda(x)=0$ on $\mathcal{A}.$
By multiplying the equations~\eqref{HJE} by $\Lambda_i>0$
and summing for $i=1,\dots,m,$ we obtain
\begin{eqnarray}\label{LTB2}
&& \frac{\partial}{\partial t}( \sum_{i=1}^m \Lambda_iu_i)
+ \sum_{i=1}^m \Lambda_i F_i(x,Du_i)
+  \sum_{j=1}^m\left(\sum_{i=1}^m \Lambda_id_{ij}\right) u_j
= \sum_{i=1}^m \Lambda_i f_i
\quad {\rm in} \ \A\times (0,+\infty),
\end{eqnarray}
in the viscosity sense. Moreover, since $u$ is Lipschitz continuous
(Proposition \ref{PrReg}),~\eqref{LTB2} holds almost everywhere.
Formally, since $F_i\geq 0,$ $D(x)^T\Lambda(x)=0$
and $f_i(x)=0$ for $x\in\mathcal{A},$
it follows that
\begin{equation*}
\frac{\partial}{\partial t}(\sum_{i=1}^m \Lambda_iu_i)\le
0 \qquad {\rm in} \ \mathcal{A}\times (0,+\infty),
\end{equation*}
and $\sum_{i=1}^m \Lambda_iu_i(\cdot,x)$ is nonincreasing in $\mathcal{A}.$
More precisely, we have: 
\begin{lem} \label{sommeui}
There exists a Lipschitz continuous $\phi: \mathcal{A}\to \R$
such that
\begin{equation}\label{LB7}
 \sum_{i=1}^m \Lambda_i u_i(\cdot ,t) \ \downarrow \
\phi \quad
\text{uniformly on $\mathcal{A}$  as $t\to +\infty.$}
\end{equation}
\end{lem}
\noindent The proof is postponed.
\smallskip

\noindent{\it Step 3. Uniform convergence of a subsequence of $u.$}
Notice that the assumptions of Theorem~\ref{thm-ergodic} hold.
It follows that there exists a solution to the ergodic problem~\eqref{HJ}
and therefore, from Proposition~\ref{PrReg} (see also Remark~\ref{rem-borne}),
$(u(\cdot,t))_{t\geq 0}$ is relatively compact in $C(\T^N)$ and there exists
a sequence $t_n\to+\infty$ such that $u(\cdot,t_n)$ converges uniformly
on $\T^N$ as $n\to +\infty.$ From~\eqref{comp-max}, we obtain that
for all $n,q\in\N,$
\begin{eqnarray*}
\mathop{\rm max}_{1\leq i\leq m}\mathop{\rm sup}_{\T^N \times [0,+\infty)}\, 
|u_i(\cdot,t_n+\cdot)-u_i(\cdot,t_q+\cdot)|
\leq 
\mathop{\rm max}_{1\leq i\leq m}\mathop{\rm sup}_{\T^N}\, 
|u_i(\cdot,t_n)-u_i(\cdot,t_q)|
\end{eqnarray*}
and therefore $(u(\cdot,t_n+\cdot))_{n}$ is a Cauchy sequence in
$W^{1,\infty}(\T^N\times [0,+\infty)).$ Thus it converges uniformly to
some function $w\in W^{1,\infty}(\T^N\times [0,+\infty)).$
By the stability of viscosity solutions, $w$
is a viscosity solution of~\eqref{HJE} (see~\cite{ bcd97,barles94, el91} for details).
\smallskip

\noindent{\it Step 4. Uniform convergence of the sequence $w(\cdot,t)$
on $\A.$} The reason of introducing $w$ in Step 3 is that, from
Lemma~\ref{sommeui}, we have
\begin{eqnarray}\label{wcstA}
\sum_{i=1}^m \Lambda_i w_i(\cdot ,t) =\phi\quad
\text{on } \mathcal{A} \text{ for all } t\in[0,+\infty). 
\end{eqnarray}
Surprisingly, this is enough to prove the convergence of each $w_i$
on $\mathcal{A}$:
\begin{lem} \label{cvwF}
The function $w(\cdot,t)$ converges
uniformly on $\mathcal{A}$ to a function $u_\infty(\cdot)$ which is
Lipschitz continuous on  $\mathcal{A}.$
\end{lem}
\noindent
The complete proof is postponed, we only outline the main ideas here.
At first, from~\eqref{wcstA}, we get
\begin{equation*}
\frac{\partial}{\partial t}(\sum_{i=1}^m \Lambda_iw_i)=
0 \qquad {\rm in} \ \mathcal{A}\times (0,+\infty).
\end{equation*}
Then, writing~\eqref{LTB2} for $w$ and using in addition $D(x)^T\Lambda(x)=0$
and $f_i=0$ on $\mathcal{A}$ 
we obtain
\begin{eqnarray*}
 \sum_{i=1}^m \Lambda_i F_i(x,Dw_i)=0\qquad {\rm in} \ \mathcal{A}\times (0,+\infty).
\end{eqnarray*}
Thus, for all $i,$ $F_i(x,Dw_i)=0$ on $\A\times (0,+\infty).$
It follows that, for fixed $x\in \A,$ the system~\eqref{HJE} reduces to
a linear differential system in $\R^m,$
\begin{eqnarray*}
\frac{\partial w}{\partial t}(x,t)+D(x)w(x,t)=0 \qquad t\in [0,+\infty).
\end{eqnarray*}
The solution is given by
\begin{eqnarray*}
w(x,t)= {\rm exp}(-tD(x))w(x,0).
\end{eqnarray*}
Since $D$ is an irreducible $M$-matrix, it has 0 as a simple eigenvalue
and all the other eigenvalues have positive real part. It follows 
that both ${\rm exp}(-tD(x))$ and $w(x,t)$ converge as $t\to +\infty.$
\smallskip

\noindent{\it Step 5. Convergence of the whole sequence $(w(\cdot,t))_{t\geq 0}$
on $\T^N.$}
Since $w$ is bounded in $\T^N\times [0,+\infty),$ we can introduce
the relaxed half-limits
\begin{eqnarray}\label{rhl}
&&   \ov w(x)= (\limsup_{t\to +\infty}\phantom{ }^*\, w)(x)
= \mathop{\rm lim}_{t\to +\infty}\;
\mathop{\rm sup} \{ w(y,s) \; : \; y\in B(x,1/t), \; s\geq t \},
\\\nonumber
&&   \underline w(x)=(\liminf_{t\to +\infty}\phantom{ }_*\,w)(x)
= \mathop{\rm lim}_{t\to +\infty}\;
\mathop{\rm inf} \{ w(y,s) \; : \; y\in B(x,1/t), \; s\geq t \},
\end{eqnarray}
where the half-limits are taken componentwise.
By the stability of viscosity solutions, $\ov w$ and
$\underline w$
are respectively a sub and a supersolution of~\eqref{HJ}
with $c=0$. From Step 4, we have 
\begin{eqnarray*}
\ov w_i = \underline w_i= (u_\infty)_i
\quad \text{ on } \mathcal{A} \text{ for all } i=1,\dots, m.
\end{eqnarray*}
It follows that
\begin{eqnarray}\label{LTB9}
 \sum_{i=1}^m \Lambda_i\ov w_i = \sum_{i=1}^m \Lambda_i\underline w_i=\phi
\qquad x\in\mathcal{A}.
\end{eqnarray}
Applying Theorem~\ref{confr} (degenerate case), thanks to~\eqref{LTB9},
we obtain $\ov w\leq \underline w$ in $\T^N$ and therefore
$\ov w=\underline w=:u_\infty$ is the unique continuous viscosity
solution of~\eqref{statioHJ} such that $\sum_{i=1}^m  \Lambda_i(u_\infty)_i=\phi$ on
$\mathcal{A}.$
This gives the convergence of $w(\cdot ,t)$ to $u_\infty$ in $C(\T^N)$ and this latter
function is Lipschitz continuous. 
\smallskip

\noindent{\it Step 6. Convergence of $(u(\cdot,t))_{t\geq 0}.$}
It remains to prove that $u(\cdot,t)$ converges to $u_\infty$ as $t\to +\infty.$
We proceed as in~\cite{bs00}. Since $u$ and $w$ are both solutions
of~\eqref{HJE}, from Proposition~\ref{PrComp}, we have
\begin{eqnarray*}
\mathop{\rm max}_{1\leq i\leq m}\mathop{\rm sup}_{\T^N} |u_i(\cdot,t_n+t)-w_i(\cdot,t)|
\leq \mathop{\rm max}_{1\leq i\leq m}\mathop{\rm sup}_{\T^N}
|u_{i} (\cdot,t_n)-w_{i} (\cdot ,0)| =o_n(1),
\end{eqnarray*}
where $o_n(1)\to 0$ as $n\to +\infty$
since $u (\cdot,t_n+\cdot)$ converges uniformly to $w(\cdot,\cdot).$
It follows that, for all $1\leq i\leq m,$ $x\in\T^N$ and $t\geq 0,$
\begin{eqnarray*}
w(x,t)-o_n(1)\leq u(x,t_n+t)\leq w(x,t)+o_n(1).
\end{eqnarray*}
Taking the relaxed half-limits in the above inequality, we obtain
\begin{eqnarray*}
u_\infty(x)-o_n(1)
= (\liminf_{t\to +\infty}\phantom{ }_* w)(x)-o_n(1)
\leq (\liminf_{t\to +\infty}\phantom{ }_* u)(x)
\end{eqnarray*}
and
\begin{eqnarray*}
(\limsup_{t\to +\infty}\phantom{ }^* u)(x)
\leq  (\limsup_{t\to +\infty}\phantom{ }^* w)(x)+o_n(1)
= u_\infty(x)+o_n(1).
\end{eqnarray*}
Sending $n$ to $+\infty$ implies
\begin{eqnarray*}
(\liminf_{t\to +\infty}\phantom{ }_* u)(x)
= (\limsup_{t\to +\infty}\phantom{ }^* u)(x)= u_\infty(x)
\end{eqnarray*}
which proves the uniform convergence of $u(\cdot,t)$ to $u_\infty$
as $t\to +\infty.$
\smallskip

\noindent{\it Step 7. All the $(u_\infty)_i$'s are equal on $\A.$}
Using the previous arguments, it is straightforward to see
that $D(x)u_\infty(x)\leq 0$ for all $x\in\A.$ 
From~\eqref{degenerate}, it follows
that $u_\infty(x)\in {\rm ker}\, D(x)$ and therefore, from
Lemma~\ref{lem-rang},
$(u_\infty)_i =(u_\infty)_j$ on $\A$ for all $i,j.$
\end{proof}

\begin{rem} \label{rmq-thm-cv} \
\begin{enumerate}

\item Let us mention that there is an easy version of the convergence
theorem when~\eqref{H3}, \eqref{H0}--\eqref{regf} and \eqref{strict-pos} holds. In this
case,~\eqref{HJ} has a unique Lipschitz continuous solution $u$ with $c=0.$
The solution of~\eqref{HJE}
is bounded and we define the relaxed half-limits as in~\eqref{rhl}.
The classical case of Theorem~\ref{confr} yields $\overline{u}=\underline{u}=:u_\infty$
and gives the convergence. Actually, in this case, $u_\infty=0.$

\item  Assuming $\mathcal{A}\not=\emptyset,$ means that
all the $f_i$'s are zero at least at a common
point (see~\eqref{aubry} and Remark~\ref{minfi}). If~\eqref{degenerate}
holds with $E_1=\T^N$
(i.e., if $\cap_i \mathcal{D}_i=\T^N$ in~\eqref{aubry}), then
we can replace this assumption with
\begin{eqnarray}
&& \exists \bar{f}\geq 0 \text{ such that } \mathop{\rm min}_{\T^N} f_i =\bar{f}
\text{ and } \mathop{\bigcap}_{1\leq i\leq m} {\rm argmin}\, f_i \not=\emptyset.
\label{nouveauA}
\end{eqnarray}
Indeed, using~\eqref{degenerate}, we recover the framework of
Theorem~\ref{LTB} by replacing $u_i(x,t)$ with
$\tilde{u}_i(x,t)=u(x,t)-\bar{f}t$ and $f_i$ with
$\tilde{f}_i-\bar{f}\geq 0.$
In a future work, we aim at studying the large time behavior of the
system~\eqref{HJE} when $\mathcal{A}=\emptyset.$

\item With the same kind of arguments as in Lemma~\ref{sommeui}, we can
prove that
\begin{eqnarray}\label{maxcv}
\mathop{\rm max}_{1\leq j\leq m} u_j(x,t)  \ \downarrow \
\tilde\phi \quad
\text{uniformly on $\mathcal{A}$  as $t\to +\infty.$}
\end{eqnarray}
The rough idea is that, if
$u_i(x,t)={\rm max}_{1\leq j\leq m} u_j(x,t)$ for $(x,t),$
then
\begin{eqnarray*}
\sum_{j=1}^m d_{ij}u_j(x,t)\geq 0
\end{eqnarray*}
and therefore, using the $i$th equation, we get
$\displaystyle \frac{\partial u_i}{\partial t}(x,t)\leq 0$
on $\mathcal{A}.$
Note that Lemma~\ref{sommeui} and~\eqref{maxcv} give an alternative
proof for the convergence theorem~\ref{LTB} when $m=2$ since
these two convergences are enough to imply the convergence
of $u_1$ and $u_2$ on $\mathcal{A}.$
This approach is used in~\cite{mt11} to obtain the convergence
result for two equations in a very similar setting.

\item Formula~\eqref{wcstA} means that 
every element $\psi$ of the $\omega$-limit set
$\o(u_0)$ given by~\eqref{omega-lim-set} satisfies
\begin{equation}\label{LT8}
\sum_{i=1}^m   \Lambda_i \psi_i=\phi\qquad {\rm on} \ \A.
\end{equation}
Moreover, since $\o(u_0)$ is positively invariant for the semigroup,  
for every $\psi\in \o(u_0)$
the restriction of  $S(t)\psi$ to $\F$ is constant in time.

\end{enumerate}
\end{rem}

\begin{ex}\label{contre-ex1}
Extensions to Theorem~\ref{LTB} are not easy to obtain. When the assumptions
of the previous theorem are not satisfied, the convergence  is not always true.
The following example is similar to the one in \cite{bs00}.
Consider
\begin{equation}\label{ex1}
\left\{  \begin{array}{lr}
   \displaystyle\frac{\partial u_1}{\partial t}+|Du_1+\a|+u_1-u_2=\a
&\text{in $\R\times (0,\infty),$}\\[3mm]
   \displaystyle\frac{\partial u_2}{\partial t}+|Du_2+\b|-u_1+u_2=\b
&\text{in $\R\times (0,\infty),$}\\[3mm]
   u_1(x,0)=u_2(x,0)=\sin(x).
   \end{array}\right.
\end{equation}
If $\a$, $\b>1$, then the unique solution of \eqref{ex1} is
$u(x,t)=(\sin(x-t),\sin(x-t))$ which clearly does not converge as $t\to \infty$.
In this case, \eqref{H0}, \eqref{regf}, \eqref{degenerate}
and~\eqref{Dcst-irr} hold but~\eqref{H1} fails and $\mathcal{A}=\emptyset.$
Moreover, notice that the assumptions of Theorem~\ref{thm-ergodic-bis}
hold. The ergodic problem has a solution $(c,v)$ with $c=(0,0)$
and $v=(C,C)$ where $C$ is any constant.
\end{ex}

\begin{proof}[Proof of Lemma \ref{sommeui}]
The proof is inspired
from the corresponding one for a scalar equation in~\cite{nr99}.
We set
\begin{eqnarray*}
U=\sum_{i=1}^m \Lambda_iu_i
\quad {\rm and} \quad
\Phi= \sum_{i=1}^m \Lambda_if_i - \sum_{j=1}^m\left(\sum_{i=1}^m \Lambda_id_{ij}\right) u_j
\end{eqnarray*}
and  $\omega(\cdot)$ is a modulus of continuity
for $\Phi(\cdot,t)$ on the compact set $\T^N.$
Note that the modulus is independent of $t$ because
of~\eqref{reg2}.
Let $x\in \mathcal{A}$ and consider the cube $B_\eps(x)=[x-\eps,x+\eps]^N$
for $\eps>0.$
Noticing that, for all $t\geq 0,$ 
$\Phi(x,t)=0$ by the very definition of $\mathcal{A}$ and $\Lambda,$
we have
\begin{eqnarray*}
 \frac{1}{\eps^n}\int_{B_\eps(x)} [U(y,t+h)-U(y,t)]dy
&=&\frac{1}{\eps^n}\int_{B_\eps(x)\times[t,t+h]}  \frac{\partial U}{\partial t}(s,y)ds\,dy\\
&\leq & \frac{1}{\eps^n} \int_{B_\eps(x)\times[t,t+h]} \Phi(y,s)ds\,dy\\
&\leq& h\omega(\eps).
\end{eqnarray*}
Sending $\eps\to 0$ and using the continuity of $U$, we get $U(x,t+h)\le U(x,t)$ for any $h>0$.
\end{proof}

\begin{proof}[Proof of Lemma \ref{cvwF}]
Arguing as in Step 2 of the proof of Theorem~\ref{LTB},
since $w$ is solution to~\eqref{HJE}, we have
\begin{eqnarray*}
&& \frac{\partial}{\partial t}( \sum_{i=1}^m \Lambda_iw_i)
+ \sum_{i=1}^m \Lambda_i F_i(x,Dw_i)
+  \sum_{j=1}^m\left(\sum_{i=1}^m \Lambda_id_{ij}\right) w_j
= \sum_{i=1}^m \Lambda_i f_i
\end{eqnarray*}
almost everywhere in $\T^N\times (0,+\infty).$

From now on, we fix $\bar x\in\A.$
Setting
\begin{eqnarray*}
W=\sum_{i=1}^m \Lambda_iw_i,\quad
\Psi = \sum_{i=1}^m \Lambda_i F_i(y,Dw_i)
\quad {\rm and} \quad
\Phi= \sum_{i=1}^m \Lambda_if_i - \sum_{j=1}^m\left(\sum_{i=1}^m \Lambda_id_{ij}\right) w_j
\end{eqnarray*}
and integrating as in the proof of Lemma~\ref{sommeui}, we have
\begin{eqnarray}\label{form000}
&& \frac{1}{\eps^n}\int_{B_\eps(\bar x)} [W(y,t+h)-W(y,t)]dy\\
&=&\frac{1}{\eps^n}\int_{B_\eps(\bar x)\times[t,t+h]}  
\frac{\partial W}{\partial t}(y,s)ds\,dy\nonumber\\
&=& \frac{1}{\eps^n}\int_{B_\eps(\bar x)\times[t,t+h]}\Phi(y,s)ds\,dy
- \frac{1}{\eps^n}\int_{B_\eps(\bar x)\times[t,t+h]}\Psi(y,s)ds\,dy. 
\nonumber
\end{eqnarray}
As in Lemma~\ref{sommeui}, since 
$\bar x\in\A,$ 
\begin{eqnarray*}
\frac{1}{\eps^n}\int_{B_\eps(\bar x)\times[t,t+h]}\Phi(y,s)ds\,dy= o_\eps (1),
\end{eqnarray*}
where $o_\eps (1)\to 0$ as $\eps\to 0$ uniformly with respect to $t.$
Since $W$ does not depend on $t$ on $\A$ from~\eqref{wcstA}, we have,
for all $y\in B_\eps(\bar x),$
\begin{eqnarray*}
|W(y,t+h)-W(y,t)|\leq |W(y,t+h)-W(\bar x,t+h)|+|W(\bar x,t)-W(y,t)|= o_\eps (1),
\end{eqnarray*}
using the uniform continuity of the $d_{ij}$'s on $\T^N$ 
and the Lipschitz continuity
of the $w_i$'s with respect to $y\in\T^N$ uniformly in $t.$
It follows from~\eqref{form000} that
\begin{eqnarray*}
\frac{1}{\eps^n}\int_{B_\eps(\bar x)\times[t,t+h]}\Psi(y,s)ds\,dy = o_\eps (1)
\end{eqnarray*}
and therefore, for all $1\leq i\leq m,$
using the uniform continuity of $\Lambda_i$
on $\T^N$ and the boundedeness of $F_i(y,Dw_i),$
we have
\begin{eqnarray}\label{Fegal0}
\frac{1}{\eps^n}\int_{B_\eps(\bar x)\times[t,t+h]}
F_i(y,Dw_i(y,s))ds\,dy= o_\eps (1).
\end{eqnarray}

Let $1\leq i\leq m$ and consider the $i$th equation:
\begin{eqnarray*}
&& \frac{\partial w_i}{\partial t}
+ F_i(y,Dw_i)
+  \sum_{j=1}^md_{ij} w_j
= f_i
\quad {\rm almost \ everywhere \ in } \ \T^N\times (0,+\infty).
\end{eqnarray*}
Integrating as above around $\bar x\in\A,$ we get
\begin{eqnarray*}
&& \frac{1}{\eps^n}\int_{B_\eps(\bar x)} (w_i(y,t+h)-w_i(y,t))dy
+ \frac{1}{\eps^n}\int_{B_\eps(\bar x)\times[t,t+h]}\sum_{j=1}^md_{ij} w_j\,ds\,dy\\
=
&&\frac{1}{\eps^n}\int_{B_\eps(\bar x)\times[t,t+h]}\frac{\partial w_i}{\partial t}
ds\,dy
+ \frac{1}{\eps^n}\int_{B_\eps(\bar x)\times[t,t+h]}\sum_{j=1}^md_{ij} w_j\,ds\,dy\\
&=&\frac{1}{\eps^n}\int_{B_\eps(\bar x)\times[t,t+h]} f_i(y)ds\,dy
- \frac{1}{\eps^n}\int_{B_\eps(\bar x)\times[t,t+h]}
F_i(y,Dw_i(y,s))ds\,dy\\
&=& o_\eps (1)
\end{eqnarray*}
from~\eqref{Fegal0} and since $\bar x\in \A.$
Using the continuity of $w_i$ and $d_{ij}$ and sending $\eps$ to 0,
we get, for all $t,h> 0,$
\begin{eqnarray*}
w_i(\bar x,t+h)-w_i(\bar x,t)+ \int_{t}^{t+h}\sum_{j=1}^md_{ij}(\bar x) w_j(\bar x,s)ds=0,
\quad 1\leq i\leq m.
\end{eqnarray*}
Recalling that $t\mapsto w(\bar x,t)$ is Lipschitz continuous, we obtain
that $w$ is solution to the linear differential system
\begin{eqnarray*}
\frac{\partial w}{\partial t}(\bar x,t)+ D(\bar{x})w(\bar{x},t)=0
\quad {\rm almost \ everywhere \ for } \ t\in (0,+\infty).  
\end{eqnarray*}
The unique solution of this system is
\begin{eqnarray*}
w(\bar{x},t)= {\rm exp}(-tD(\bar{x}))w(\bar{x},0).
\end{eqnarray*}
Since $D(\bar{x})$ is an irreducible $M$-matrix, $0$ is a simple
eigenvalue and all the nonzero eigenvalues have a positive real
part. It follows that there exists a matrix $A(\bar x)$ such that
${\rm exp}(-tD(\bar{x}))= A(\bar x)+O({\rm e}^{-rt})$
where $r>0$ is the smallest real part of the nonzero eigenvalues.
Therefore, 
\begin{eqnarray*}
w(\bar{x},t) \mathop{\to}_{t\to +\infty} A(\bar x)w(\bar{x},0)
=\mathop{\rm lim}_{n\to +\infty} A(\bar x)u(\bar{x},t_n) =: u_\infty(\bar{x}).
\end{eqnarray*}
Since $w\in W^{1,\infty}(\T^N\times [0,+\infty)),$ we obtain
that $u_\infty\in  W^{1,\infty}(\T^N).$
\end{proof}
\section{The control-theoretic interpretation}\label{sec:control}

At least when the coefficients $d_{ij}$ of the coupling matrix $D$
are constant, we can give an interpretation of our system
of Hamilton-Jacobi equations~\eqref{HJE} in terms of optimal
control of hybrid systems with pathwise deterministic
trajectories with random switching.
We do not give the proofs here, we refer the readers to
Fleming and Zhang~\cite{fz98},
Yong and Zhou~\cite{yz98} and the references therein.

Consider the controlled random evolution process
$(X_t,\nu_t)$ with dynamics
\begin{equation}\label{REV}
\left\{
\begin{array}{l}
\dot X_t =  b_{\nu_t}(X_t ,a_t), \ \ t> 0,\\
(X_0,\nu_0)=(x,i) \in \T^N\times \{1,\dots,m\},
\end{array}
\right.
\end{equation}
where the control law  $a:[0,\infty)\to A$ is a measurable function
($A$ is a compact subset of some metric space),
$b_i\in L^\infty(\T^N\times A; \R^N),$  satisfies
\begin{eqnarray}
|b_i(x,a)-b_i(y,a)|\leq C|x-y|,\qquad x,y\in\T^N, \ a\in A, \ 1\leq i\leq m.
\end{eqnarray}
For every $a_t$ and matrix of probability transition
$G=(\g_{ij})_{i,j}$ satisfying $\sum_{j\not= i}\g_{ij}=1$
for $i\not= j$ and $\g_{ii}=-1,$
there exists a solution $(X_t,\nu_t),$ where $X_t:[0,\infty)\to \T^N$
is piecewise $C^1$ and $\nu(t)$  is a continuous-time Markov chain
with state space $\{1,\dots,m\}$ and probability transitions given by
\[
\P\{\nu_{t+\Delta t}=j\,|\,
\nu_t=i\}=\g_{ij}\Delta t+O(\Delta t)
\]
for $j\neq i.$

We introduce the value functions of the optimal control problems
\begin{equation}\label{Value}
    u_i(x,t)=\inf_{a_t\in L^\infty([0,t],A)}\E_{x,i}\{\int_0^t f_{\nu_s}(X_s)ds
+u_{0,\nu_t} (X_t)\},
\quad i=1,\dots m,
\end{equation}
where $\E_{x,i}$ denote the expectation of a trajectory starting at
$x$ in the mode $i,$ $f_i, u_{0,i}:\T^N\to \R$ are continuous and
$f_i\geq 0.$

It is possible to show that the following dynamic programming principle
holds:
\begin{eqnarray*}
u_i (x,t)= \mathop{\rm inf}_{a_t\in L^\infty([0,t],A)} \E_{x,i} \{
\int_0^t f_{\nu_s} (X_s)ds + u_{\nu_h}(X_h,t-h)
\} \qquad 0<h\leq t.
\end{eqnarray*}
Then the functions $u_i$ satisfy the system
\begin{equation*}
\left\{
  \begin{array}{ll}\displaystyle
\frac{\partial u_i}{\partial t}+ \mathop{\rm sup}_{a\in A}
-\langle b_i(x,a), Du_i\rangle
+\sum_{j\not= i}\g_{ij}(u_i-u_j)= f_i
   & (x,t)\in  \T^N\times (0,+\infty), \\[5pt]
  u_{i}(x,0)=u_{0,i}(x)&x\in \T^N,
  \end{array}
i=1, \cdots m,
\right.
\end{equation*}
which has the form \eqref{HJE} by setting
$F_i(x,p)= \mathop{\rm sup}_{a\in A}
-\langle b_i(x,a), p\rangle$ and $d_{ii}=\sum_{j\not= i} \g_{ij}=1$
and $d_{ij}= -\g_{ij}$ for $j\not=i.$

The assumptions~\eqref{H3}, \eqref{H0} and \eqref{regf}
are clearly satisfied
and~\eqref{H1} holds if the following controllability assumption is satisfied: 
for every $i,$
there exists $r>0$ such that for any $x\in\T^N$, the ball $B(0,r)$ is contained in $\overline{\textrm{co}}\{b_i(x,A)\}$. Moreover,
$\cap_{1\leq i\leq m}\mathcal{D}_i=\T^N$ and $\A=\F.$

Assuming \eqref{aubry}, \eqref{Dcst-irr}, we
obtain that Theorem~\ref{LTB} holds. Roughly speaking, it means that
the optimal strategy is to drive the trajectories towards a point $x^*$ of
$\A$ and then not to move anymore (except maybe a small time before $t$).
This is suggested by the fact that all the $f_i$'s have minimum 0
at $x^*$ and, at such point, the running cost is 0.
Now, if \eqref{aubry} does not hold anymore, things appear to be
more complicated.
We hope to come back to this issue in a future work.



\end{document}